\documentclass[11pt, reqno]{article}

\usepackage{fullpage}
\usepackage{enumerate}
\usepackage{setspace}

\usepackage{amsmath,amsfonts,amssymb,graphics,amsthm, comment}
\usepackage{hyperref}
\hypersetup{
    colorlinks=true,
    linkcolor=blue,
    citecolor=red,
    urlcolor=blue,
    pdfborder={0 0 0}
}    
\usepackage{graphicx}
\usepackage{xcolor}

\usepackage{cleveref}

  \crefname{theorem}{Theorem}{Theorems}
  \crefname{thm}{Theorem}{Theorems}
  \crefname{lemma}{Lemma}{Lemmas}
  \crefname{lem}{Lemma}{Lemmas}
  \crefname{remark}{Remark}{Remarks}
  \crefname{prop}{Proposition}{Propositions}
  \crefname{proposition}{Proposition}{Propositions}
\crefname{notation}{Notation}{Notations}
\crefname{claim}{Claim}{Claims}
  \crefname{defn}{Definition}{Definitions}
  \crefname{corollary}{Corollary}{Corollaries}
  \crefname{section}{Section}{Sections}
  \crefname{figure}{Figure}{Figures}
  \crefname{question}{Question}{Questions}
  \crefname{exercise}{Exercise}{Exercises}
    \crefname{assumption}{Assumption}{Assumptions}

\newtheorem{thm}{Theorem}[section]

\newtheorem{lemma}[thm]{Lemma}
\newtheorem{corollary}[thm]{Corollary}
\newtheorem{prop}[thm]{Proposition}

\newtheorem{proposition}[thm]{Proposition}

\newtheorem{question}[thm]{Question}

\numberwithin{equation}{section}

\theoremstyle{definition}
\newtheorem{remark}[thm]{Remark}

\def\cW{\mathcal{W}}
\def\cV{\mathcal{V}}
\def\cU{\mathcal{U}}

\def\cR{\mathcal{R}}

\def\cH{\mathcal{H}}
\def\cG{\mathcal{G}}

\def\cE{\mathcal{E}}
\def\cD{\mathcal{D}}
\def\cC{\mathcal{C}}
\def\cB{\mathcal{B}}
\def\cA{\mathcal{A}}

\def \ve {\varepsilon}

\def\P{\mathbb{P}}
\def\E{\mathbb{E}}

\def\R{\mathbb{R}}
\def\Z{\mathbb{Z}}
\def\N{\mathbb{N}}

\def\R{\mathbb{R}}

\def\1{\mathbf{1}}

\def  \p- {p\textunderscore}

\DeclareMathOperator{\Bern}{Bern}

\DeclareMathOperator{\w} {\mathsf{w}}
\DeclareMathOperator{\f}{\mathsf{f}}

\DeclareMathOperator {\bP} {\boldsymbol{ \mathsf{P}}}
\def \fiid {\textsf{fiid }}
\def \cGm {\cG^{\bullet}_{\mathsf m}}

\usepackage{extarrows}

\def\proj{\mathsf{proj}}
\date{}

\title{Uniform even subgraphs and graphical representations of Ising as factors of i.i.d.}
\author{Omer Angel \thanks{University of British Columbia. Research supported in part by NSERC. Email: angel@math.ubc.ca} \and Gourab Ray \thanks{University of Victoria. Research supported in part by NSERC 50311-57400. Email:gourabray@uvic.ca } \and Yinon Spinka\thanks{University of British Columbia. Research supported in part by NSERC. Email: yinon@math.ubc.ca}}

\begin{document}
\maketitle
\begin{abstract}
We prove that the Loop $O(1)$ model, a well-known graphical expansion of the Ising model, is a factor of i.i.d.\ on unimodular random rooted graphs under various conditions, including in the presence of a non-negative external field. As an application we show that the gradient of the free Ising model is a factor of i.i.d.\ on unimodular planar maps having a locally finite dual. The key idea is to develop an appropriate theory of local limits of uniform even subgraphs with various boundary conditions and prove that they can be sampled as a factor of i.i.d.\ Another key tool we prove and exploit is that the wired uniform spanning tree on a unimodular transient graph is a factor of i.i.d.\ This partially answers some questions posed by Hutchcroft \cite{hutchcroft2020continuity}.
\end{abstract}

\section{Introduction}\label{sec:introduction}
The Loop $O(1)$ model has received much attention in recent years, mostly due to its intimate relation with the Ising model and its FK-Ising representation.
The model has also been instrumental to much of the recent progress in understanding the Ising model (see e.g.\ \cite{duminil_graphical} for a brief overview of this topic).
Given a statistical physics model on an infinite graph, a natural question is whether it is a factor of i.i.d., or in other words, whether it can be represented as an automorphism equivariant function of some collection of i.i.d.\ random variables associated to the vertices of the graph.
The goal of this article is to prove that the Loop $O(1)$ model is indeed a factor of i.i.d.\ in several cases.
This partially answers some questions of Hutchcroft~\cite{hutchcroft2020continuity}, who exploited such properties to prove the continuity of the phase transition in the Ising model on nonamenable groups.
Along the way, we also investigate relevant questions about expressing Uniform even subgraphs and Uniform spanning forests as factors of i.i.d.

 Let $G = (V,E)$ be a locally finite graph. 
 We consider (combined edge and site) percolation configurations $\eta \in \{0,1\}^{E \cup V}$. An \textbf{even percolation configuration} is an $\eta$ such that $\partial \eta = \emptyset$, where
 \begin{equation}
 \partial \eta = \left\{v \in V: \eta(v) + \sum_{e \in E : v \in e} \eta(e) \text{ is odd}\right\}.\label{eq:boundary}
 \end{equation}
The \textbf{Loop $O(1)$ model} on a finite graph $G$ is a probability measure on even percolation configurations, or more generally, on percolation configurations which are even except on a given boundary set. Given parameters $x,y \in \R_+$ and a boundary set $B \subset V$, it is defined as
\begin{equation}
\bP^B_{G,x,y}(\eta) = \frac{1}{Z^B_{G,x,y}}x^{\#\{e \in E: \eta(e)=1\}}y^{\#\{v \in V:\eta(v)=1\}}\1_{\{\partial \eta \subset B \}} \1_{\{\eta_B\equiv1\}}; \qquad \eta \in \{0,1\}^{E \cup V},\label{eq:free_loop}
\end{equation}
where $Z^B_{G,x,y}$ is the partition function.
While the model is defined for all $x,y\in\R_+$, the parameter range $x,y\in[0,1]$ is most interesting due to the connection with the Ising model.
The parameter~$x$ is related to the temperature of the Ising model, and the parameter $y$ is related to the intensity of the external field in the Ising model.
When $y = 0$ and $B=\emptyset$, the Loop $O(1)$ model is supported on \emph{edge} percolation configurations in which the degree of every vertex is even. Such configurations are identified with \textbf{even (spanning) subgraphs} of $G$,
which play a central role in this article. When $B=\emptyset$, we sometimes drop it from the notation, and similarly when $y=0$.

%

One reason why the Loop $O(1)$ measures are hard to analyze is that they lack certain monotonicity properties~\cite{klausen2020monotonicity}, which are enjoyed by the Ising model as well as its FK-Ising representation.
These monotonicity properties can be exploited to obtain results about them as factors of i.i.d.\ \cite{adams1992folner,van1999existence,haggstrom2002coupling,harel2018finitary,OW73}.
The key fact used to study the Loop $O(1)$ model in this article is that it can be realized as a uniform even subgraph of the FK-Ising model. Indeed, for $x=1,y=0,B=\emptyset$, it is the uniform even subgraph of $G$ itself.

Suppose now that $G$ is an infinite, locally finite, connected graph, and let $G_n$ be an exhaustion of $G$, i.e.\ an increasing sequence of finite subgraphs with $\cup G_n = G$.
We consider two natural loop $O(1)$ measures on $G_n$. The first is simply the loop $O(1)$ measure on $G_n$ with ``free boundary'', namely, $\bP_{G_n,x,y}^\emptyset$. The second is the loop $O(1)$ measure with ``wired boundary'', defined as follows:
Let $G_n^{\w}$ be the graph obtained by gluing all the vertices of $G$ which are not in $G_n$ into a single vertex $\Delta$, and removing the resulting self loops at $\Delta$. Then $\bP_{G_n^{\w},x,y}^{\w}=\bP_{G_n^{\w},x,y}^{\{\Delta\}}$ the loop $O(1)$ measure on $G_n$ with ``wired boundary''.\footnote{It is convenient sometimes to introduce a ghost vertex $v^*$ and attach an edge $\{v,v^*\}$ if and only if $\eta(v) =1$ and $v \neq \Delta$, glue together $\Delta$ and $v^*$, and erase the self loop. Then in the above definition, the Loop $O(1)$ measure is supported on subgraphs of even degree in this enhanced graph.}

It is known \cite[Theorem 2.3]{aizenman2015random} that both $\bP_{G_n,x,y}$ and $\bP_{G^{\w}_n,x,y}^{\w}$ converge weakly to limiting percolation measures on $G$.
This is essentially a consequence of the convergence of Ising correlation functions.
The limit measures are called the \textbf{free Loop $O(1)$} and the \textbf{wired Loop $O(1)$} measures (with parameters $x$ and $y$) and are denoted by $\bP_{G,x,y}^{\f}$ and $\bP_{G,x,y}^{\w}$ respectively.
In \cref{sec:loop}, we provide an alternate proof of the existence of the limits in an explicit almost sure sense which does not rely on the convergence of Ising correlation functions. 

\medskip

A process on a (connected and infinite) transitive graph is said to be a \textbf{factor of i.i.d.}\ (fiid) if it can be written as a automorphism-equivariant measurable function of a collection of i.i.d.\ random variables attached to the vertices of $G$ (see \cref{sec:coding_def} for a detailed definition).
In this article we work with a generalization of this notion for processes on random rooted graphs.
A rooted graph is a pair $(G,\rho)$ consisting of a locally finite, connected graph $G$ and a vertex $\rho$ called the root.
The notion of a factor of i.i.d.\ in this setup generalizes to what we call a \textbf{graph factor of i.i.d.} The ideas involved with this definition are borrowed from \cite{holroyd2017finitary,timar2021nonamenable} (see \cref{sec:coding_def}).
This is a natural extension of the notion of a factor of i.i.d., where the equivariance to automorphisms of $G$ is replaced by equivariance to isomorphisms of marked rooted graphs:
Given a random rooted graph $(G, \rho)$ and a stochastic process $X=(X_v)_{v \in V(G)}$ on it, the triplet $(G, \rho, X)$ is a graph factor of i.i.d.\ if there is an i.i.d.\ collection $\Xi=(\Xi_v)_{v \in V(G)}$ and a measurable function $F$ defined on isomorphism classes of marked rooted graphs such that $X_v = F(G,v,\Xi)$ for each $v\in G$.
We remark that if $(G, \rho)$ is a deterministic transitive graph, then the notion of a graph factor of i.i.d.\ coincides with the usual notion of factor of i.i.d.\ (see \cref{lem:relation_factor}).
In this paper we need to work with the case where $G$ is a supercritical FK-Ising cluster in some larger graph $H$, which is not transitive even when $H$ is.

A common regularity assumption in this setting is that the random rooted graph is \textbf{unimodular}.
This roughly means that it obeys certain distributional symmetry properties (see \cref{sec:unimodular} for a definition).
Our main results require unimodularity. In particular, there are nonunimodular vertex-transitive graphs for which our results do not apply; see \cref{rmk:nonunimodular} where we point out the steps in the proof which break down. 
A typical non-random example of a unimodular graph is the Cayley graph of a finitely generated group (e.g., the free group with 2 generators).
Typical random examples are percolation clusters on Cayley graphs, and the Poisson-Voronoi tesselation or Delaunay triangulations (appropriately rooted).
Unimodular random graphs have also received much attention recently; we refer to the work of Aldous and Lyons \cite{AL_unimodular} which laid down the foundations of this topic. See also \cite{curiennotes} for an excellent exposition of this topic.

\medskip

We proceed to describe our main results.
We begin with the wired Loop $O(1)$ measure, for which the result is most complete.

\begin{thm}\label{thm:main_wired}
Let $(G, \rho)$ be a unimodular, random rooted graph with finite expected degree of $\rho$. Let $x,y \in [0,1] $ and given $(G, \rho)$, let $\eta$ be sampled from $\bP^{\w}_{G, x,y}$.
\begin{itemize}
\item If $(x,y) \neq (1,0)$ then $(G, \rho, \eta)$ is a graph factor of i.i.d.
\item If $(x,y)  = (1,0)$ then   $(G, \rho, \eta)$ is a graph factor of i.i.d. if and only if $(G, \rho)$ is almost surely not two ended.
\end{itemize}
\end{thm}
We refer to \cref{sec:deterministic_even} for the definition of an end of a graph which plays a crucial role in the analysis.
Let us illustrate with an example why $(G, \rho, \eta)$ is not a graph factor if $(x,y)=(1,0)$ and $(G, \rho)$ is two ended. First note that $\eta$ is the ``wired uniform even subgraph'' of the graph $(G, \rho)$ itself. A crucial idea used in this article is that in a wired uniform even subgraph, every cycle or bi-infinite path can be independently included (added modulo 2) according to a fair coin flip (we refer to \cref{sec:proj} for more details). A trivial example is given by $\Z$, where the wired uniform even subgraph is either empty or everything with equal probability. A more interesting example is the ladder graph $\Z \times \{0,1\}$, where it is impossible to decide on the presence of the bi-infinite paths corresponding to the two copies of $\Z$ as a factor of i.i.d. (in fact, the wired uniform even subgraph is not even ergodic).

We now turn to the free Loop $O(1)$ measure, for which our results are less comprehensive. We need several definitions before stating the result. A graph $G$ is (vertex) \textbf{amenable} if there exists a sequence of finite subsets $V_n \subset V$ such that
\[ \lim_{n \to \infty} \frac{|\partial V_n|}{|V_n|} =0, \]
where $\partial V_n$ is the set of vertices in $V_n$ with at least one neighbour not in $V_n$ and $|\cdot|$ denotes cardinality.
There is a related notion of amenability more suitable to random graphs, called \textbf{invariant amenability}, where the finite sets $V_n$ must be chosen as root clusters of invariant percolation processes.
We refer to \cref{sec:coding_def} for relevant definitions;
for now we quickly mention that by \cite[Theorems 5.1 and 5.3]{benjamini1999group}, a vertex-transitive unimodular graph is amenable if and only if it is invariantly amenable.

A \textbf{planar map} is a proper embedding of a graph into a simply connected open subset $S$ of $\R^2$ viewed up to orientation preserving homeomorphisms. A rooted planar map is a planar map with a fixed root vertex. Similar to unimodular random rooted graphs, there are also analogous notions of unimodular random rooted maps and  map factors of i.i.d.\ We refer to \cref{sec:unimodular} for more details on this topic. 

A \textbf{cycle} in $G$ is a sequence of vertices $(v_0,v_1, \ldots, v_n = v_0)$ with $v_i \sim v_{i+1}$ for all $0 \le i \le n-1$.  A \textbf{geodesic cycle} in $G$ is a finite cycle such that for any two vertices in the cycle, the shorter path between them along the cycle is a geodesic path in $G$.

We also refer to \cref{sec:FK_ising} for the definition of the FK-Ising model.

\begin{thm}\label{thm:main_free}
Let $x,y \in [0,1]$, let $(G, \rho)$ be a unimodular random rooted graph with finite expected degree of $\rho$ and given $(G, \rho)$, let $\eta$ be sampled from $\bP^{\f}_{G, x,y}$. Then $(G, \rho, \eta)$ is a graph factor of i.i.d.\ in each of the following cases:
\begin{enumerate}[a.]
\item $y>0$.
\item $(G,\rho)$ is almost surely invariantly amenable.
\item $\omega$ has only finitely many geodesic cycles through any vertex almost surely, where $\omega$ is a sample from the free FK-Ising meaure on $(G, \rho)$ with parameter $p = \frac{2x}{1+x}$ (and with no external field).
\end{enumerate}
Furthermore, if $(M,\rho)$ is a unimodular random rooted planar map with finite expected degree of $\rho$ and $\eta$ is sampled from $\bP^{\f}_{M, x,y}$, then $(M,\rho,\eta)$ is a map factor of i.i.d.
\end{thm}

Regarding the third condition, we do not know whether it is satisfied in general for $x<1$. However, for $x=1$, it seems that it is possible to construct Cayley graphs with infinitely many geodesic cycles through a vertex (suitable versions of the so-called \emph{Gromov monsters}). 
See \cref{sec:open} for further discussion.

\medskip

Given a stochastic process derived from a statistical physics model on general a graph, the question of whether it is a factor of i.i.d.\ has received much attention in recent times.
Although there has been a lot of progress in the case of amenable graphs \cite{Adams92,adams1992folner,haggstrom2000propp,spinka2018finitaryising,spinka2018finitarymrf,RS_19,sly2019stationary,ray2020proper},
the nonamenable setting remains a rather unexplored territory;
see Bowen \cite{bowen2010measure} for a general result in this direction and Lyons \cite{lyons2017factors} for a survey of results on a tree.
See also \cite{lyons2011perfect,backhausz2017spectral,gerencser2019mutual,backhausz2018correlation,csoka2015invariant,harangi2015independence,nam2020ising}  for other relevant results.
We elaborate a bit on this now and explain how our results have some consequences for the gradient of Ising model on planar graphs. 

Recall that the Ising model on a finite graph $G = (V,E)$ is a probability measure on spin configurations $\sigma\in\{\pm1\}^V$ defined by
\begin{equation}
  \mathbf I_{G, \beta} (\sigma) \propto \exp\left[\beta \sum_{x \sim y} \sigma_{x}\sigma_y\right] \propto \exp\left[-2\beta \sum_{x \sim y} \1_{\sigma_{x} \neq \sigma_y}\right].
\end{equation}
Here $\beta$ is the inverse temperature. We always assume $\beta \ge 0$ in this article which corresponds to the ferromagnetic regime.
Given an infinite, locally finite graph $G$, one can take an exhaustion $(G_n)_{n \ge 1}$ and take the free and wired limits of the Ising measure, similar to the loop $O(1)$ definition.
Denote the resulting measure on $G$ by $\mathbf I^{\f}_{G, \beta}$ for the free measure and $\mathbf I^{+}_{G, \beta}$ or $\mathbf I^{-}_{G, \beta}$ for the wired measure obtained as a limit with $+1$ or $-1$ on the boundary vertex respectively.
One can also consider \textbf{gradient} measures on $\{0,1\}^{E}$ by taking the pushforward of these measures under the map $\{\sigma_v\}_{v \in V} \mapsto \{\1_{\sigma_u \neq \sigma_v}\}_{\{u,v\} \in E}$.
Denote these measures by $\mathbf G^{\f}_{G, \beta}$ and $\mathbf G^{\w}_{G, \beta} = \mathbf G^{+}_{G, \beta} = \mathbf G^{-}_{G, \beta}$ (the two are equal, since the Ising measure is invariant under spin flip).

It is known that the wired Ising measure $\mathbf I^{\pm}_{G, \beta}$ is always a factor of i.i.d.\ (this was shown for $\Z^d$ in \cite{OW73}, for amenable groups in \cite{adams1992folner}, and finally for arbitrary graphs in \cite{haggstrom2002coupling}).
On an amenable graph, the free Ising measure $\mathbf I^{\f}_{G, \beta}$ is always equal to the mixture $\frac12(\mathbf I^{+}_{G, \beta} + \mathbf I^{-}_{G, \beta})$ of the wired measures~\cite{raoufi2020translation}, and in particular, it is a factor of iid if and only if there is a unique Gibbs measure (i.e., $\mathbf I^{+}_{G, \beta}=\mathbf I^{-}_{G, \beta}$). In contrast, the question of whether the free Ising measure $\mathbf I^{\f}_{G, \beta}$ is a factor of i.i.d.\ on a nonamenable graph is delicate even in the simplest case of regular trees.
On a $d$-regular tree, the uniqueness regime (where $\mathbf I^{+}_{G, \beta}$, $\mathbf I^{-}_{G, \beta}$, $\mathbf I^{\f}_{G, \beta}$ all coincide) is $\tanh(\beta) \le 1/(d-1)$, and in particular, the free Ising measure is a factor of i.i.d.\ in this regime.
On the other hand, if $\tanh(\beta) > 1/\sqrt{d-1}$ (the so-called reconstruction regime), then the free Ising measure is \emph{not} a factor of i.i.d.\ (see~\cite{lyons2017factors}).
Recently, Nam, Sly and Zhang \cite{nam2020ising} showed that for large enough $d$, the regime where it is a factor of i.i.d.\ extends beyond the uniqueness threshold, to $\tanh(\beta)<c(d)$ for some $c(d)>1/(d-1)$.
The exact location of the transition for being a factor of i.i.d.\ remains unknown.

On the other hand, it is easy to see that $\mathbf G^{\f}_{G, \beta}$, the gradient of the free Ising measure, is always a factor of i.i.d.\ on a tree (as it is itself actually i.i.d.).
This raises the natural question of whether the free gradient measure is always a factor of i.i.d.\ on a general graph.
In this context, the Loop $O(1)$ model is relevant: The finite dimensional distribution of Loop $O(1)$ can be written as the expectation of an explicit functional of the gradient of Ising (see \cite{aizenman2015random}). 

\medbreak

Suppose $M$ is a planar map with a locally finite dual (i.e., all faces have finite degree).
Let $M^{\dagger}$ denote the dual map of $M$, and for every edge $e \in E(M)$, let $e^{\dagger}$ denote the dual edge crossing it.
Given a spin configuration $\sigma$ on the vertices of $M$, we define a bond percolation on $M^{\dagger}$ by declaring an edge $e^{\dagger}$ open if and only if $\sigma_u \neq \sigma_v$ where $e  = \{u,v\}$ (or equivalently, if the gradient of $\sigma$ is 1 at $e$).
It is easy to see that the resulting percolation configuration on $M^{\dagger}$ is necessarily even, and in fact has the same law as the Loop $O(1)$ model with parameter $x = e^{-2\beta}$.
To be more precise, $\mathbf I^{\f}_{M, \beta}$ is pushed forward to $\mathbf P^{\w}_{M^{\dagger}, x}$ and $\mathbf I^{\pm}_{M, \beta}$ is pushed forward to $\mathbf P^{\f}_{M^{\dagger}, x}$.
Given a unimodular planar map $(M,\rho)$ with finite expected degree of the root, there is a natural way to define a unimodular dual map $(M^{\dagger}, \rho^{\dagger})$.
We refer to \cite[Section 2.5]{AHNR2} for more details of this construction.
Since the wired Loop $O(1)$ model on the dual map can be sampled as a factor of i.i.d.\ (by \cref{thm:main_wired}) and the mapping from the Loop $O(1)$ in the dual to the Ising gradient in the primal is a local deterministic map, we get the following corollary:

\begin{corollary}\label{thm:gradient}
  Let $(M,\rho)$ be a unimodular random rooted planar map with finite expected degree of $\rho$, and which has a locally finite dual almost surely.
  Let $\gamma$ be a sample from $\mathbf G^{\f}_{M, \beta}$ for some $\beta\ge 0$. Then $(M,\rho,\gamma)$ is a map factor of i.i.d.
\end{corollary}

See also \cite{RS_19} for related results in $\Z^d$.

\medbreak

As mentioned earlier, a key step in the proofs of \cref{thm:main_wired,thm:main_free} is to obtain the uniform even subgraph of the FK-Ising clusters as a factor of i.i.d.
This will follow from a general result in \cref{sec:uni_even}.
One important step is to obtain a one-ended spanning tree or forest as a factor of i.i.d.
In the transient case, the candidate we choose is the wired uniform spanning forest (WUSF), which is the limit of uniform measures on wired spanning trees along an exhaustion of the graph (see, e.g., \cite{lyons2017probability}).
This is a well-known and studied object, which among other things, is known to satisfy the aforementioned one-endedness property~\cite{BLPS_USF,AL_unimodular,H15a}.
However, as far as we can tell, the fact that it is a factor of i.i.d.\ is not known, and so we prove the following:

\begin{thm}\label{thm:UST}
  Let $(G,\rho)$ be a random rooted graph that is almost surely transient.
  Then the wired uniform spanning forest on $G$ is a graph factor of i.i.d.
\end{thm}

Note that \cref{thm:UST} does not require the underlying graph to be unimodular.
The proof of \cref{thm:UST} relies on the cycle popping algorithm due to Wilson, modified to the setting of infinite graphs. In this setting, the induction step in the proof of Wilson is replaced by a transfinite induction step.
The case of a recurrent random rooted graph remains open (\cref{qn:recurrentUST}).

\medbreak
\paragraph{Organization.}
We start with some background material in \cref{sec:background} (definitions of codings in \cref{sec:coding_def}, of unimodular random graphs in \cref{sec:unimodular}, of FK-Ising in \cref{sec:FK_ising} and a description of the coupling between FK-Ising and Loop $O(1)$ in \cref{sec:loopO1}). In \cref{sec:uni} we build the theory for uniform even subgraphs, and in particular, in \cref{sec:proj} we describe a way to define and generate uniform even subgraphs of infinite graphs. In \cref{sec:loop}, we apply this theory to generate the Loop $O(1)$ model from FK-Ising.
In \cref{sec:USFfactor} we show how to generate the wired uniform spanning forests as a factor (\cref{thm:UST}; this is independent of everything else). Finally in \cref{sec:loop_factor} we combine everything together to prove \cref{thm:main_wired,thm:main_free}. We finish with some open questions and discussions in \cref{sec:open}.
\paragraph{Acknowledgement:} We thank Tom Hutchcroft several inspiring comments on an earlier draft of the paper.

\section{Background}\label{sec:background}

\subsection{Coding definitions}\label{sec:coding_def}
We present here precise definitions of factors and graph factors. Let us begin with the more standard notion of factors of i.i.d.\ on a fixed transitive graph.
Let $G = (V,E)$ be a vertex-transitive graph and let $\Gamma$ be the automorphism group of $G$ (though sometimes $\Gamma$ is taken to be a transitive subgroup, e.g., the group of translations when $G=\Z^d$).
Let $X = (X_t)_{t \in E\cup V }$ be a stochastic process taking values in $\R^{E \cup V}$.
We say that $X$ is a \emph{factor of i.i.d.}\ (\fiid for short) if there exists a collection of i.i.d.\ random variables $\Xi = (\xi_{v})_{v \in V}$ and an automorphism equivariant-function $\varphi: \R^{V} \mapsto \R^{E \cup V}$ such that $\varphi(\Xi) = X$ almost surely.
Here, automorphism-equivariant means that for any $\gamma \in \Gamma$, $$\gamma \varphi(\Xi) = \varphi(\gamma \Xi) \quad \text{ a.s.}$$
where for $\omega \in \R^{E \cup V}$,  $(\gamma \omega)_{t \in E \cup V} := (\omega_{\gamma^{-1}(t)})_{t \in E \cup V}$ is the diagonal action of $\Gamma$.
The same notation is used for processes $X$ defined only on the edges, or only on vertices of $G$.

As mentioned in the introduction, we will also need a notion of factors of i.i.d.\ on random rooted graphs, which we borrow from \cite{timar2021nonamenable,holroyd2017finitary}.
A \textbf{rooted graph} is a locally finite, connected graph $G$ where one vertex $\rho$ is specified as the root.
We will also consider \textbf{marked rooted graphs} which are triplets $(G, \rho, m)$ where $m:E\cup V \mapsto \Omega$ is a mapping to a Polish space $\Omega$.
Two (marked) rooted graphs are equivalent if there is a graph isomorphism between them that maps the root to the root and preserves the marks.

Let $\cG^\bullet$ (resp.\ $\cGm$) denote the space of equivalence classes of rooted graphs (resp.\ marked rooted graphs) endowed with the \textbf{local topology}:
informally, two marked rooted graphs are close if for some large $R$ the balls of radius $R$ are isomorphic as graphs, and the restriction of the marks to the balls are uniformly close
(see \cite{benjamini2011recurrence}).
We similarly use $\cG_{\mathsf m}^{\bullet \bullet}$ to denote the space of equivalence classes of doubly rooted marked graphs $(G, \rho_1,\rho_2, m)$.

A \textbf{graph factor} is a measurable map which maps a marked rooted graph to the \emph{same} rooted graph but with a new marking which does not depend on the location of the root.
Formally, a graph factor is a function $\varphi: \cGm \mapsto \cGm$ such that $\varphi((g,x,m)) = (g,x,m')$ for all $(g,x,m)$ and such that $\varphi$ is invariant under rerooting in the sense that if $\varphi((g,x,m)) = (g,x,m')$ then $\varphi((g,y,m)) = (g,y,m')$ for any other vertex $y$. 
Equivalently, a function $\varphi: \cGm \mapsto \cGm$ is a graph factor if and only if there exist measurable functions $F: \cGm \mapsto \Omega$ and $F': \cG_{\mathsf m}^{\bullet \bullet} \mapsto \Omega$ such that
\[ \varphi((g,\rho,m)) = (g,\rho,m') \qquad \text{and}\qquad \begin{array}{ll} m'(u) = F((g,u,m)) &\text{for all vertices }v, \\ m'(e) = F'((g,u,v,m)) &\text{for all edges }e=\{u,v\}, \end{array} \]
where $F'$ is invariant to transposing the two roots in the sense that $F'((g,x,y,m))=F'((g,y,x,m))$.

A sample $(G, \rho) $ from a probability measure on $\cG^\bullet$ is called a \textbf{random rooted graph}.
Similarly, a random rooted, marked graph $(G, \rho, M)$ is a sample from a probability measure on $\cGm$.
An \textbf{i.i.d.-marked random rooted graph} $(G, \rho , \Xi)$ is a random rooted graph $(G, \rho)$ with i.i.d.\ markings $\Xi = \{\xi_v\}_{v\in V(G)}$ on the vertices.\footnote{There is a canonical choice of a representative for each element in $\cG^\bullet$, allowing to make proper sense of this; see e.g.\ \cite[Section 2]{AL_unimodular}.}
We may assume the markings are distributed as Uniform$[0,1]$.

We call a random, rooted, marked graph $(G, \rho, M)$ a \textbf{graph factor of i.i.d.}\  if there exists a graph factor map $\varphi$ and an i.i.d.-marked random rooted graph  $(G, \rho , \Xi)$ such that 
\[ \varphi((G, \rho, \Xi))= (G, \rho, M) \text{ a.s.} \]
We sometimes say that $M$ is a graph factor of i.i.d.\ when the underlying random rooted graph $(G,\rho)$ is clear from context, admitting a slight abuse of terminology.
See also \cite[Definition 6]{timar2021nonamenable}.
 
It is a common alternative to attach the marks to the edges instead of vertices of the graph, and sometimes marks are attached to both vertices and edges.
Since we can create many i.i.d.\ Uniform$[0,1]$ random variables as a measurable function of a single Uniform$[0,1]$, it is easy to see that the different resulting notions of graph factors of i.i.d.\ are equivalent in all these cases.
Similarly, we may replace the Uniform$[0,1]$ variables by a countable sequence of independent variables at each vertex, which may not Uniform$[0,1]$ but rather any other distribution of our choosing.
 
The following lemma (which is perhaps well known but for which we have not found a reference) shows that the notion of graph factor of i.i.d.\ indeed generalizes the classical notion of factor of i.i.d.\ on a fixed transitive graph.

\begin{lemma}\label{lem:relation_factor}
  Let $G$ be a vertex transitive graph, let $\rho$ be a fixed vertex in it, and let $X$ be a stochastic process on it.
  Then $X$ is a factor of i.i.d.\ if and only if $(G, \rho, X)$ is a graph factor of i.i.d.
\end{lemma}
\begin{proof}
  Let $X$ be a factor of some i.i.d.\ process $\Xi$ and let $\varphi$ be the factor map.
  Now define a graph factor map $\psi: \cG^{\bullet}_{\mathsf m} \mapsto \cG^{\bullet}_{\mathsf m}$ such that $\psi((G, \rho, m)) =  (G, \rho, \varphi(m))$ (and define $\psi((g,x, m))$ arbitrarily if $(g, x)$ is not equivalent to $(G, x)$).
  Note that $\psi$ is well defined.
  Indeed, suppose that $(G, \rho, m) $ and $(G, \rho', m')$ are equivalent as marked rooted graphs and let $f:V(G) \mapsto V(G)$ be a root and mark preserving isomorphism, so that $f(\rho) = \rho'$ and $f(m) = m'$.
  Then the equivariance of $\varphi$ implies that $\varphi(m') = \varphi(f(m)) = f(\varphi(m)).$ So $(G, \rho, \varphi(m))$ and $(G, \rho',\varphi(m'))$ are equivalent as marked rooted graphs with $f$ being an isomorphism between them.
  It is clear from the definition that $\psi$ does not depend on the location of the root.
  Thus, $(G, \rho, X) = \psi ((G, \rho, \Xi))$ and hence $(G, \rho, X)$ is a graph factor of i.i.d.\

  Conversely suppose that $(G, \rho, X)$ is a graph factor of i.i.d.\ so that $(G,\rho,X)=\varphi((G,\rho,\Xi))$ for some graph factor $\varphi$ with $(G, \rho, \Xi)$ being an i.i.d.\ marked random rooted graph.
  Let $F$ and $F'$ be as in the second definition of graph factor.
  Now define $\psi :\R^{V} \to \R^{E \cup V}$ by $\psi(m)_v = F((G,v,m))$ and $\psi(m)_{\{u,v\}} = F'((G,u,v,m))$.
  Then $\psi$ is equivariant, since for an automorphism $\gamma$ of $G$,
\[
\psi(\gamma m)_v = F((G,v,\gamma m)) = F((G,\gamma^{-1}(v),m)) = \psi(m)_{\gamma^{-1}(v)} = (\gamma \psi(m))_v ,\]
and similarly $\psi(\gamma m)_{\{u,v\}} = (\gamma \psi(m))_{\{u,v\}}$ for any edge $\{u,v\}$.
This implies in particular that $\psi(\Xi)$ is well defined and that it equals $X$.
This completes the proof, except for the subtle technical issue that the i.i.d.\ process $\Xi$ comes from the random marked rooted graph $(G,\rho,\Xi)$, and is thus, strictly speaking, not an i.i.d.\ process on the fixed graph $G$ (to conclude that $X$ is a factor of i.i.d., we need to express it as a function of an i.i.d.\ process on $G$). The problem is that since $(G,\rho,\Xi)$ is an equivalence class of marked rooted graphs, $\Xi$ defines a realization of an i.i.d.\ process on $G$, but only up to automorphisms of $G$ (e.g., if $G=\Z$ and $\rho=0$, then for almost every realization of $(G,\rho,\Xi)$ there are two different ways to choose marks on $\Z$, which are related to one another by a reflection around 0).
To bypass this, we let $\Xi'$ be i.i.d.\ Uniform $[0,1]$ on the vertices of $G$ and note that $[(G,\rho,\Xi')] \in \cG^{\bullet}_{\mathsf m}$ has the same distribution as $(G,\rho,\Xi)$.
Consequently, $\psi(\Xi')$ has the same distribution as $X$, so that we can put $\Xi'$ and $X$ in the same probability space so that $\psi(\Xi') = X$ almost surely.
This shows that $X$ is a factor of i.i.d.
\end{proof}


\subsection{Unimodular random graphs and maps} \label{sec:unimodular}
We continue this section with the definition of unimodular random rooted marked graphs and some of its consequences.
These are fairly standard, and the expert reader may wish to skip this section.
Recall the definition of marked (doubly) rooted graphs from \cref{sec:coding_def}. 
A random rooted marked graph $(G, \rho, M) \in \cG_{\mathsf m}^{\bullet}$ is \textbf{unimodular} if it satisfies the mass transport principle, i.e., for all measurable $f:\cG_{\mathsf m}^{\bullet\bullet} \mapsto [0, \infty)$,
\begin{equation*}
\E\left(\sum_{x \in V(G)} f((G,u,x,M))\right) = \E\left(\sum_{x \in V(G)} f((G,x,u,M)) \right).
\end{equation*}
Informally, we regard $f((G,x,y,M))$ as the mass sent from $x$ to $y$, and then unimodularity means that the expected total mass sent by $u$ equals the expected total mass received by $u$.
An unmarked rooted graph $(G,\rho)$ is unimodular if $(G,\rho,0)$ is unimodular (i.e. $G$ with a trivial marking).

We now introduce a notion of amenability which is relevant in the unimodular setting.
Let $(G, \rho)$ be a unimodular random rooted graph. We say that $\omega : E(G) \mapsto \{0,1\}$ is an \textbf{invariant percolation} on $(G, \rho)$ if $(G, \rho, \omega)$ is a unimodular random rooted marked graph. 
Given a finite set of vertices $W$ in $G$, we define
\begin{equation} i_{\sf E}(W) := \frac{|\partial_{\sf E} W|}{|W|}, \label{eq:edge_boundary}
\end{equation}
where $\partial_{\sf E} W$ denotes the set of edges in $G$ with exactly one endpoint in $W$ (and the other endpoint outside $W$).
Let $\omega_{\rho}$ denote the connected component of $\rho$ in the graph $(V,\omega)$. Define
\begin{equation}
\iota((G,\rho)) :=\inf\left\{ \E(i_{\sf E}(V(\omega_\rho))): \substack{\text{\normalsize $\omega$ is an invariant percolation on $(G, \rho)$ such that}\\\text{\normalsize every connected component of $\omega$ is a.s. finite}} \right\}.\label{eq:inv_edge_boundary}
\end{equation}
If $\iota((G,\rho))=0$ then we say that $(G, \rho)$ is \textbf{invariantly amenable.}. 


We now record some properties of unimodular random rooted graphs which we need in this article.
For this we need some additional definitions.
For a percolation configuration $\omega \in \{0,1\}^{E \cup V}$ and an edge or vertex $t \in E \cup V$, define $\omega_t$ to be the same as $\omega$ but with value 0 at $t$.
Similarly define $\omega^t$ to be the same $\omega$ but with value 1 at $t$.
For $\cA \subset \{0,1\}^{E \cup V}$, we also define $\cA_t := \{(g,\rho, \omega_{t}): (g,\rho, \omega) \in \cA) \}$, and $\cA^t$ similarly.
We also allow $t$ to be a measurable function of $(g,\rho,m)$, with the same definitions.
We say that a probability measure $\mu$ on $\cG^{\bullet}_{\mathsf m}$ is \textbf{deletion tolerant} (resp.\ \textbf{insertion tolerant}) if for any event $\cA \subset \cG^{\bullet}_{\mathsf m}$ with $\mu(\cA)>0$ and any Borel measurable map $t:(G, \rho, m) \mapsto E(G) \cup V(G)$, we have that $\mu(\cA_t)>0$ (resp.\ $\mu(\cA^t) >0$).
We say that an event $\cA \subset \cG^{\bullet}_{\mathsf m} $ is \textbf{invariant to re-rooting} if for any $x,y$ we have $(g,x,m) \in \cA \iff (g,y,m) \in \cA$.
We say that $(G, \rho, m)$ is \textbf{ergodic} if any event $\cA$ which is invariant to re-rooting has probability either 0 or 1.
It is known \cite[Theorem 4.7]{AL_unimodular} that a unimodular random rooted graph is ergodic if and only if its law is \textbf{extremal}, that is, it cannot be expressed as a nontrivial convex combination of unimodular random rooted graphs.
Using Choquet theory, every unimodular random rooted graph can be written as a convex combination of its extremal components.

\begin{lemma}\label{unimod_properties}
Let $(G, \rho)$ be a unimodular random rooted graph with law $\mu$.
\begin{enumerate}[a.]
\item \label{mark_unimod} The random rooted graph $(G,\rho,\Xi)$ with i.i.d.\ marks is unimodular.
\item \label{factor_unimod}If $(G,\rho,m)$ is unimodular, then any graph factor of it is also unimodular.
\item \label{reroot} If $(G, \rho, m)$ is unimodular and $\cA$ is an event invariant to re-rooting with $\mu(\cA)>0$, then $(G,\rho,m)$ conditioned on $\cA$ is unimodular 
\item \label{root_see_all}If $(G,\rho,m)$ is unimodular and $\mu(\cA)=1$ for some measurable event in $\cG^{\bullet}_{\mathsf m}$, then almost surely $(G,v,m) \in \cA$ for all $v$. 
\end{enumerate}
\end{lemma}

\begin{proof}
  The first two  items are easy exercises in definitions which we leave to the reader.
 For the third see also \cite[Exercise 15]{curiennotes}.
  The last is \cite[Proposition 11]{curiennotes}.
\end{proof}

\begin{lemma}\label{lem:ergodic}
  If a random rooted graph $(G,\rho)$ is unimodular, ergodic and almost surely infinite, then the i.i.d.\ marked random rooted graph $(G,\rho,\Xi)$ is ergodic.
  If a marked random rooted graph $(G,\rho,m)$ is ergodic, then so is any graph factor of it. 
\end{lemma}

\begin{proof}
  This is a straightforward application of the definitions which we leave to the reader.
\end{proof}

\begin{lemma}\label{unimod_perc}
  Let $(G, \rho)$ be a unimodular random rooted graph and let $\omega$ be an invariant percolation on it. 
  \begin{enumerate}[a.]
  \item \label{tolerant_condition} If $\omega$ is insertion (or deletion) tolerant and $\cA$ is an event invariant to re-rooting and has positive probability, then $(G, \rho, \omega)$ conditioned on $\cA$ is insertion (or deletion) tolerant and unimodular.
\item\label{root_cluster} $(\omega_\rho, \rho)$ is unimodular. 

\item \label{inv_amenable}$\iota ((\omega_\rho, \rho)) \le \iota ((G, \rho))$. 

\item \label{rec_one_ended}If $\rho$ has finite expected degree and $\omega$ is both insertion and deletion tolerant, then almost surely either every infinite cluster of $\omega$ is recurrent and one ended, or every infinite cluster of $\omega$ is transient.
%
\end{enumerate}
\end{lemma}

\begin{proof}
  The proof of \cref{tolerant_condition} is a straightforward application of the definitions, see for example \cite[Lemma 6.8]{AL_unimodular} for a stronger statement.
  The proof of \cref{root_cluster} is also an easy consequence of the definition.
  For \cref{inv_amenable}, consider any invariant percolation $\tilde \omega$ with a.s. finite clusters on $G$.
  Consider the mass transport $f(x,y)$ where $x$ sends a total mass of $\deg_{G\setminus\tilde\omega}(x)$, uniformly distributed to all vertices $y$ in the cluster of $x$ (and $0$ if $x,y$ are in distinct or infinite clusters).
  The mass transport principle for this $f$ yields
  \[
    \E(i_{\sf E}(V(\tilde \omega_\rho))) = \E(\deg_{G\setminus \tilde \omega} \rho).
  \]
  Furthermore, if given $(G, \rho)$, we sample $\omega$ and $\tilde \omega$ conditionally independently then $(G, \rho, (\omega, \tilde \omega))$ is unimodular and
  \[
    \E(\deg_{\omega \setminus \tilde \omega} (\rho)) \le  \E(\deg_{G \setminus \tilde \omega} (\rho)).
  \]
  Taking the infimum over percolations $\tilde\omega$ with finite components,
  it follows that $\iota ((\omega_\rho, \rho)) \le \iota ((G, \rho))$, completing the proof.

  For \cref{rec_one_ended}, using \cref{tolerant_condition} we can assume that $(G, \rho)$ is ergodic. Recall that insertion tolerant, ergodic, invariant percolation clusters are indistinguishable  \cite[Theorem 6.15]{AL_unimodular} in the sense that every invariant property is satisfied either by all infinite clusters a.s.\ or by none a.s. We apply this to two invariant properties: transience and the number of ends of the percolation clusters. We combine this with \cite[Theorem 8.9]{AL_unimodular} which states that every invariantly amenable unimodular graph with finite expected degree has at most two ends almost surely.  Using indistinguishability, \cref{root_cluster} and \cite[Proposition 8.7]{AL_unimodular}, we conclude that either every infinite cluster is recurrent with 2 ends a.s., or every infinite cluster is recurrent with one end a.s., or every infinite cluster is transient a.s.
Using deletion tolerance, it is easy to conclude that the first case cannot occur since if it did, the root cluster being infinite and two ended must have positive probability. But by deletion tolerance, we can decompose the root cluster into two one-ended components with a finite cost (we skip the details here) which leads to a contradiction. This concludes the proof.
\end{proof}

We now discuss briefly the notion of unimodular planar maps which is involved in final part of \cref{thm:main_free} and in \cref{thm:gradient}. 
Recall that an embedding of a graph onto a surface $S$ is \textbf{proper} if no two edges cross each other, every compact set $S$ intersects only finitely many edges and vertices, and every face (connected component of the complement of edges and vertices) is homeomorphic to a disc.
For our purposes, we assume $S$ is a simply connected subset in the plane $\R^2$. A \textbf{planar map} is a graph equipped with a proper embedding onto $S$.
We say two such embeddings onto surfaces $S$ and $S'$ are equivalent if one can be mapped to the other by an orientation preserving homeomorphism between $S$ and $S'$.
A rooted planar map comes with an assigned root vertex (sometimes the root is taken to be a directed edge but it will not be relevant for us).
Given a rooted planar map $(M, \rho)$ and its corresponding rooted graph $(G, \rho)$, there is a natural marking which assigns to each vertex $v$ a cyclic permutation $\sigma_v$ on the edges incident to $v$ giving the order they eminate from $v$.
It is well known the marked graph $(G,\rho,\sigma)$ determines the map uniquely~\cite{lando2004graphs}.
We say that $(M,\rho)$ is a unimodular planar map if $(G,\rho,\sigma)$ is a unimodular marked rooted graph.
We refer to \cite[Section 2]{AHNR2} for a comprehensive treatment of this topic.
One can also analogously define the notion of a \textbf{map factor of i.i.d.} where the statements concerning measurability with respect to rooted or bi-rooted graphs are replaced by measurability with respect to the space of rooted or bi-rooted maps.

\subsection{The FK-Ising model}\label{sec:FK_ising}

The Fortuin--Kasteleyn Ising model (a special case of the random cluster model) is a well-known percolation model in statistical physics with diverse uses and applications.
The FK-Ising model with parameters $p,p_h \in [0,1]$ (edge weight $p$ and external field $p_h$) on a finite graph $G = (V,E)$ is the probability measure $\phi^{}_{G,p,p_h}$ given by
\begin{equation}
  \phi^{}_{G,p,p_h}(\omega) \propto
  \left(\frac{p}{1-p}\right)^{\# \{e \in E:\omega(e)=1\}} \left(\frac{p_h}{1-p_h}\right)^{\#\{v \in V: \omega(v)=1\}}2^{k(\omega)}; \qquad \omega \in \{0,1\}^{E \cup V},
  \label{eq:RC}
\end{equation}
where $k(\omega)$ is the number of clusters (connected components) in the graph $(V,\{e \in E:\omega(e)=1\})$ that do not intersect $\{v \in V : \omega(v)=1\}$, or equivalently, if we connect all vertices with $\omega(v)=1$ to an external vertex $\Delta$ by an edge, then $k(\omega)+1$ is the number of clusters formed by these edges together with the edges $\{e \in E:\omega(e)=1\}$.

These measures enjoy a monotonicity property known as FKG (positive correlations of increasing events).
One consequence of this is the existence of weak limits of the finite-volume measures:
Let $G_n$ be any increasing sequence of finite subgraphs of $G$ which exhaust $G$.
Then the weak limit of $\phi^{}_{G_n,p,p_h}$ exists, and does not depend on the exhaustion.
The limit is called the \textbf{free FK-Ising measure}, and is denoted by $\phi^{\f}_{G,p,p_h}$.
Similarly, given $G_n$, we can identify all the vertices not in $G_n$ into a single vertex $\Delta$, and remove the self loops at $\Delta$.
Call the resulting graph $G_n^{\w}$ (which could have multiple edges) and let $\phi^{\w}_{G_n^{\w},p,p_h}$ be the measure $\phi_{G_n^{\w},p,p_h}(\cdot \mid \omega(\Delta)=1)$. 
The weak limit of  $\phi^{\w}_{G_n^{\w},p,p_h}$ also exists (again, this is a consequence of FKG), is called the \textbf{wired FK-Ising measure}, and is denoted by $\phi^{\w}_{G,p,p_h}$.
The FKG property extends to these two limiting measures, and was exploited crucially by H{\"a}ggstr{\"o}m--Jonasson--Lyons~\cite{haggstrom2002coupling} and Harel--Spinka~\cite{harel2018finitary} to prove versions of the following theorem.

\begin{thm}[\cite{haggstrom2002coupling,harel2018finitary}]\label{thm:HS}
Let $(G, \rho)$ be a random rooted graph, let $p,p_h \in [0,1]$ and given $(G, \rho)$, let $\omega^{\f} \sim \phi^{\f}_{G,p,p_h}$ and $\omega^{\w} \sim \phi^{\w}_{G,p,p_h}$. Then $(G, \rho, \omega^{\f})$ and $(G, \rho, \omega^{\w})$ are graph factors of i.i.d..
\end{thm}

Specifically, when $(G, \rho)$ is a fixed graph and $p_h=0$, this was essentially shown (though not explicitly stated) in \cite{haggstrom2002coupling}.
When $(G,\rho)$ is a fixed quasi-transitive graph and for any $p_h$, this follows from the results in~\cite{harel2018finitary}.
The same argument works for random graphs, and we include a sketch for completeness.
Indeed, an inspection of the latter proof reveals that it relies only on monotonicity (it is based on Glauber dynamics and coupling from the past), and not on any particular properties of the underlying graph, and hence extends to the setting of random rooted graphs leading to a graph factor of i.i.d.

As mentioned earlier, the Loop $O(1)$ model lacks the FKG properties enjoyed by FK-Ising~\cite{klausen2020monotonicity}.
This is one of the main reasons that the techniques in \cite{haggstrom2002coupling,harel2018finitary} are not directly applicable for the Loop $O(1)$ model.
However, the intimate connection between the two models that we discuss next allows us to make use of the results for the FK-Ising model.


\subsection{Coupling the Loop $O(1)$ and FK-Ising models}
\label{sec:loopO1}

In this section we explain a useful coupling between the Loop $O(1)$ model and the FK-Ising model.
This first appeared in recent work of Aizenman, Duminil-Copin, Tassion and Warzel \cite[Theorem 3.2]{aizenman2019emergent}.
The coupling illustrates that by taking a union of a Loop $O(1)$ model with an i.i.d.\ Bernoulli edge percolation, with carefuly set parameters, we can obtain the FK-Ising model.
In particular, the Loop $O(1)$ model can be realized as  a subset of the FK-Ising model. 

Let $\phi^B_{G,p,p_h}$ denote the law of $ \omega \sim \phi_{G,p,p_h}$ conditioned on $\omega_B \equiv 1$.

\begin{proposition}\label{prop:coupling}
  Let $G=(V,E)$ be a finite graph and let $B \subset V$. Let $x,y \in [0,1]$ and let $\eta$ be sampled from $\bP^B_{G,x,y}$.
  Let $X$ be an independent Bernoulli percolation process on $E \cup V$ with success parameter $x$ on $E$ and $y$ on $V$.
  Define
  \[\eta'(t) = \eta(t) \vee  X(t); \qquad t\in E \cup V. \]
  Then $\eta'$ has law $\phi^B_{G,p,p_h}$ where
  \[ p = \frac{2x}{1+x} \qquad\text{and}\qquad p_h = \frac{2y}{1+y}.
  \]
  Furthermore, given $\eta'$, the conditional law of $\eta$ is uniform on the set of all percolation configurations $\omega\le \eta'$  such that $\partial \omega \subset B$. \end{proposition}


Some version of this is present in \cite[Theorem 3.2]{aizenman2019emergent} and an earlier version can be found in  \cite[Theorem 3.1]{grimmett2007random}, both in the special case $p_h=0$.
Since the proof is short and crucial for this article and we could not find a reference which involves the external field as well, we provide a proof for completeness.

\begin{proof}
 
  To simplify formulas, for $t \in E \cup V$, we write
  \[
    x_t = \begin{cases}
      x = \tanh(\beta)& \text{if $t \in E$} \\
      y =\tanh(h) & \text{if $t \in V$}
    \end{cases}.
  \]
  Recall from~\eqref{eq:free_loop} that
  \begin{equation}
    \P(\eta=\omega) = \frac{1}{Z^B_{G,x,y}} \prod_{t \in E \cup V} \big( x_t \1_{\{\omega(t)=1\}} + \1_{\{\omega(t)=0\}}\big) \1_{\{\partial \omega \subset B, \omega_B \equiv 1\}}.
  \end{equation}
  From this we see that the joint law of $(\eta,\eta')$ can be written as 
  \begin{equation}
    \P(\eta=\omega, \eta' = \omega') = \frac{1}{Z^B_{G,x,y}} \prod_{t \in E \cup V} \big( x_t \1_{\{\omega'(t)=1\}}+(1-x_t)\1_{\{\omega'(t) =0\}}\big) \1_{\{\omega \le \omega', \partial \omega \subset B, \omega_B \equiv 1\}}.
    \label{eq:coupling1}
  \end{equation}
  Since the term in the parenthesis does not depend on $\omega$, denoting $\cE(\omega'):=\{\omega \le \omega': \partial \omega \subset B, \omega_B \equiv 1\}$, the marginal law of $\eta'$ can thus be written as
  \begin{align*}
    \P(\eta' = \omega')
    &= \frac{\left|\cE(\omega') \right|}{Z^B_{G,x,y}}
      \prod_{t \in E \cup V}\big( x_t \1_{\{\omega'(t)=1\}}+(1-x_t)\1_{\{\omega'(t) =0\}}\big).
  \end{align*}
  The calculation boils down to counting the number of elements in $\cE(\omega')$ when $\omega'_B \equiv 1$ (otherwise $\cE(\omega')$ is empty).
 Let $G^*$ be the graph obtained from $G$ by introducing a ghost vertex $v^*$, connecting it to every vertex in $V \setminus B$, and then identifying all vertices in $B$ with it (keeping any multiple edges and self loops which may have been created). Any $\tau \in \{0,1\}^{E \cup V}$ with $\tau_B\equiv 1$ defines a spanning subgraph $\tau^{*}$ of $G^*$ whose edges consist of $\{ e \in E : \tau(e)=1\}$ and $\{ \{v,v^*\} : v \in V \setminus B, \tau_v=1\}$. Note that the map $\tau \mapsto \tau^*$ is then a bijection between $\cE(\omega')$ and the set of all even subgraphs of $(\omega')^{*}$.\footnote{This relies on the fact that the number of odd degree vertices in any finite graph is even, so that $v^*$ cannot be (the only vertex) of odd degree in $\omega^*$.}
  Now we use the fact that the number of even subgraphs of any finite graph $H$ with $m$ connected components is given by $2^{|E(H)|-|V(H)|+m}$.
  Applying this to $H=(\omega')^*$, we get that $|\cE(\omega')|=2^{|\omega'|-|V|+k(\omega')}$, where $|\omega'| = \sum_{t \in E \cup V} \omega'_t$ and where $k(\omega')$ was defined in~\cref{sec:FK_ising} (note that $|E(H)|=|\omega'|-|B|$, $|V(H)|=|V|-|B|+1$ and $m=k(\omega')+1$). Thus,
  \begin{equation}
    \P(\eta' = \omega') = \frac{1}{Z^B_{G,x,y}} 2^{k(\omega') + |\omega'|-|V|} \prod_{t \in E \cup V} \big( x_t \1_{\{\omega'(t)=1\}}+(1-x_t)\1_{\{\omega'(t) =0\}}\big) \1_{\{\omega'_B \equiv 1\}}, \label{eq:coupling2}
  \end{equation}
  which gives that
 \[ \P(\eta' = \omega') \propto 2^{k(\omega')} \prod_{t \in E \cup V} \big(q_t \1_{\{\omega'_t =1\}} +(1-q_t)\1_{\{\omega'_t =0\}}\big) \1_{\{\omega'_B \equiv 1\}}, \]
  where $q_t= p$ if $t \in E$ and $q_t = p_h$ if $t \in V$. Thus this law is the same as $\phi^B_{G,p,p_h}$ as desired.
  The fact that given $\eta'$, the conditional law of $\eta$ is that of a uniform percolation configuration whose boundary is contained in $B$ is clear from the expression in~\eqref{eq:coupling1} since the weight on the right-hand side is independent of $\omega \subset \omega'$.
\end{proof} 

\begin{remark}\label{rmk:ghost}
We will use \cref{prop:coupling} in two situations.
The first is in order to generate a sample $\eta \sim \bP_{G,x,y}$ from the free Loop $O(1)$ measure. Indeed, \cref{prop:coupling} (used with $B=\emptyset$) tells us that such a sample $\eta$ can be obtained by first sampling $\omega' \sim \phi_{G,p,p_h}$ and then sampling a uniform even subgraph from the enhanced graph $(\omega')^*$ described in the proof of \cref{prop:coupling}.

In the second situation, we wish to generate a sample $\eta \sim \bP^{\w}_{G^{\w}_n,x,y}$ from the wired Loop $O(1)$ measure (recall that $G^{\w}_n$ is obtained from $G$ by identifying all vertices outside of $G_n$ into a single vertex $\Delta$, and that $\bP^{\w}_{G^{\w}_n,x,y}$ is simply $\bP^B_{G^{\w}_n,x,y}$ with $B=\{\Delta\}$). In this case, \cref{prop:coupling} (used with $B=\{\Delta\}$) tells us that $\eta$ can be obtained by first sampling $\omega' \sim \phi^{\w}_{G^{\w},p,p_h}$ and then sampling a uniform even subgraph of the enhanced graph $(\omega')^*$ as described in the proof of \cref{prop:coupling}.
\end{remark}

\section{Generating uniform even subgraphs}\label{sec:uni}

In light of \cref{prop:coupling}, sampling uniform even subgraphs of various random graphs is relevant for this work, and will be a key tool that we exploit.
In this section we study the ``uniform'' even subgraph of an infinite graph which can be obtained as a weak limit of uniform even subgraphs on finite graphs with free or wired boundary conditions, and discuss several ways of sampling it.

We begin by discussing uniform even subgraphs of finite graphs in \cref{sec:uni_even}. We then discuss deterministic properties of even subgraphs of infinite graphs in \cref{sec:deterministic_even}. Finally, we discuss uniform even subgraphs of infinite graphs in \cref{sec:proj} and consequences for the Loop $O(1)$ model in \cref{sec:loop}.

\subsection{Uniform even subgraph of finite graphs}\label{sec:uni_even}

Let $G = (V,E)$ be a finite graph with vertex set $V$ and edge set $E$ (we allow multiple edges).
A spanning subgraph $H$ of $G$ consists of the vertex set $V$ and an edge set $E(H) \subset E$ (all subgraphs are implicitly assumed to be spanning).
Given two subgraphs $H_1$ and $H_2$, define the addition operation $\oplus$ as $H_1 \oplus H_2 = (V, E(H_1) \Delta E(H_2))$.
Given a collection of subgraphs $\cH$ which is closed under the $\oplus$ operation, we can think of $(\cH, \oplus)$ as a vector space over $\Z_2$ (the cyclic group of order $2$).
More precisely, consider the vector space $\cV:= (\{0,1\}^E, \oplus)$ where $\omega \oplus \omega' = ((\omega_e + \omega_e') \mod 2)_{e \in E}$.
Then $(\cH, \oplus)$ can be seen as a subspace of $\cV$.
We say that a collection of subgraphs $\cU$ \textbf{generates} $\cH$ if $\textsf{span}(\cU) = \cH$, where \textsf{span} denotes the linear span in this vector space.
 
As noted, we are primarily concerned with the space $\cE = \cE(G)$ consisting of all even subgraphs of $G$.
The space $\cE$ is the kernel of the degree mod 2 operator from $\{0,1\}^E$ to $\{0,1\}^V$, and as such $\cE$ is closed under $\oplus$, and is a linear subspace of $\cV$.
It is easy to show that the set of all simple cycles in $G$ generates $\cE$.
It is also not hard to see that this need not be a basis for $\cE$.
The following lemma provides a useful way to generate a uniformly random element from $\cE$.
We use the usual summation notation $\sum$ as a shortand for iterating the operation $\oplus$.
  
\begin{lemma}\label{prop:cycle_finite}
  If $\cB = \{b_1,\dots,b_m\}$ spans a finite vector space $\mathbb{V}$ over $\Z_2$, and $\{\ve_i\}_{i=1}^m$ are independent $\Bern(1/2)$ random variables, then $L=L(\cB) := \sum \ve_i b_i$ is a uniform element of $\mathbb{V}$.
\end{lemma}

\begin{proof}[Sketch of proof]
  If $\cB$ is a basis the claim holds since $L$ is a bijection between $\mathbb{V}$ and $(\ve_i)$.
  Otherwise write $\cB$ as a disjoint union of a basis $\cB_1$ and extra vectors $\cB_2$, and write accordingly $L=L_1+L_2$.
  Then $L_1$ is uniform in $V$, and thus so is $L_1+L_2$ regardless of the law of $L_2$.
\end{proof}


\begin{corollary}\label{C:cycle_finite}
  If $\cB = \{C_i\}_i$ is a generating set of cycles in $G$, and $\ve_i$ are independent $\Bern(1/2)$ random variables, then $L := \sum \ve_i C_i$ is a uniform even subgraph of $G$.
\end{corollary}

\subsection{Even subgraphs of infinite graphs}\label{sec:deterministic_even}

The goal of this subsection is to extend \cref{C:cycle_finite} to generate ``uniform'' even subgraphs of infinite graphs.
There are two questions at hand here.
Firstly, does the limit of uniform even subgraphs of some exhaustion by finite subgraphs exist, and if so, in which sense?
Secondly, is it possible to directly generate a sample from the limiting probability measure in a nice way?

Let $G$ be a infinite, locally finite graph (we do not require $G$ to be connected).
Since we deal with an external field in the context of Ising and Loop $O(1)$, we introduce a new special vertex $v^*$ called the \emph{ghost vertex}, and allow any subset of the vertices of $G$ to be neighbors of $v^*$, perhaps with finitely many multiple edges between $v^*$ and any given vertex of $G$.
Let $G^*$ be this new graph.
Thus, $G^*$ is locally finite except possibly at the ghost vertex $v^*$, where the degree could be infinite.
We define the \textbf{wired cycle space} to be
$$
\cE^{\w} := \Big\{ \gamma \subset E(G^*): \text{every vertex other than $v^*$ has even degree in $\gamma$} \Big\},
$$
and we define the \textbf{free cycle space} to be
$$
\cE^{\f} := \overline{\mathsf{span}}\Big\{ \gamma \subset E(G^*): \gamma \text{ is a finite cycle}\Big\},
$$
where $\overline{\mathsf{span}}$ is the closure of all finite sums in the topology of pointwise convergence on $\{0,1\}^{E(G^*)}$. 
In the absence of a ghost vertex (when $v^*$ is isolated in $G^*$), the wired cycle space is simply the space of all even subgraphs. A \textbf{ray} is an infinite edge-disjoint path in $G^*$, i.e., a sequence of vertices $(v_0, v_1, \ldots)$ such that $v_{i}$ is adjacent to $v_{i+1}$ in $G^*$ for all $i \ge 0$ and all the (directed) edges $(v_i, v_{i+1})$ are distinct. It is also not hard to see that, in general, the wired cycle space is the $\overline{\mathsf{span}}$ of all finite cycles, bi-infinite paths, and rays starting at $v^*$.
Clearly, $\cE^{\f} \subset \cE^{\w}$, but in general, $\cE^{\f} \neq \cE^{\w}$, even when the ghost vertex is isolated. This is clearly demonstrated by the graph $\Z$ or by the slightly less trivial example of the ladder graph, where a bi-infinite path corresponding to a single copy of $\Z$ belongs to $\cE^{\w}$ but not to $\cE^{\f}$. See \cite{diestel2005cycle} or \cite[Chapter 8]{diestel2005graph} for some topological interpretations of such cycle spaces in infinite graphs.


There is an important connection between the wired and free cycle spaces and the space of ends of a graph. Let us recall the definition of an end of a graph.  Two rays $r$ and $s$ are equivalent if there exists a third ray $t$ which intersects $r$ and $s$ infinitely often. It can be checked that this defines an equivalence relation on rays.  An \textbf{end} of a graph is an equivalence class of rays. We say a graph is \textbf{one-ended} if it has exactly one end. For example, $\Z^2$ is one-ended, $\Z$ and the ladder graph each have two ends, while the infinite binary tree has $2^{\aleph_0}$ many ends (the ends of a tree can be identified with rays starting at a fixed vertex). If $H$ is a subgraph of $G^*$, then there is a natural map which maps the ends of $H$ to the ends of $G^*$, by sending the equivalence class of a ray $r$ in $H$ to its equivalence class in $G$ (if two rays are equivalent in $H$, then they are equivalent in $G$, so that this is well defined).
We say that $H$ is \textbf{end faithful} if this map is a bijection.
For example, $\Z \times \{0\}$ is not end faithful to $\Z^2$, but $\Z \times \N$ is.
We now state a useful result about the existence of an end faithful spanning tree of a graph.

\begin{thm}[{Halin \cite{halin1964}, see \cite[Theorem 2.2]{diestel1992end} for an English version}]\label{halin}
Every connected graph has an end faithful spanning tree.
\end{thm}

Just like in the finite case, it is useful to have a nice generating set for the wired and free cycle spaces.
We say that a countable collection $\cC$ of subgraphs is \textbf{locally finite} if every edge in $G^*$ is contained in only finitely many elements of $\cC$.
We say that $\cC$ \textbf{generates} $\cH$ if $ \overline{\mathsf{span}} (\cC) = \cH$.   
The existence of a locally finite generating set for $\cE^{\w}$ was shown in~\cite[Proposition 2.4]{grimmett2007random} in the case when $v^*$ is isolated.
The following lemma is a simple extension of this to our more general setting.

\begin{lemma}\label{lem:gen_set_GJ}
  There exists a locally finite generating for $\cE^{\f}$ consisting of finite cycles.
  There exists a locally finite generating for $\cE^{\w}$ consisting of finite cycles, bi-infinite paths and rays starting from the ghost vertex.
\end{lemma}

\begin{proof}
  Let us first consider $\cE^{\w}$.
  Let $\{e_1,e_2,\dots\}$ be an arbitrary ordering of the edges of $G^*$.
  For each $k$, let $\cH_k$ be the set of all $H \in \cE^{\w}$ which are disjoint from $\{e_1,\dots,e_{k-1}\}$.
  If there is no $H\in \cH_k$ with $e_k\in H$, then let $C_k = \emptyset$.
  If such $H\in\cH_k$ exists, let $C_k$ be a cycle, bi-infinite path, or ray through $e_k$ which does not pass through any $e_i$ with $i<k$.
  Such $C_k$ exists since $H$ contains $e_k$ and has all even degrees, and so must contain such a $C_k$.
  Clearly the resulting $\{C_k\}$ is locally finite and generates $\cE^{\w}$.
  (The proof that $\{C_i\}_i$ is locally finite and generating is exactly the same as in \cite[Proposition 2.4]{grimmett2007random}.)

  For $\cE^{\f}$, the construction is similar, except that we need $C_{k}$ to be a finite cycle.
  To see that such a finite cycle exists, let $H\in\cH_k$ be a non-empty graph containing $e_k$ and no $e_i$ with $i<k$.
  We can approximate $H$ by a finite sum of finite cycles $H'$ that also includes $e_k$ and no $e_i$ with $i<k$.
  However, a finite sum of cycles is necessarily a sum of disjoint cycles, and so $H'$ includes the desired $C_k$.  
\end{proof}
%

The problem (for our purposes) with the construction above is that it relies on an arbitrary ordering of the edges.
We now address the deterministic question of how to generate ``good'' generating sets, both for the free and wired cycle spaces, which can potentially be constructed in an invariant manner.
In the next two lemmas, we describe a method to construct a locally finite generating set for $\cE^{\w}$, and also for $\cE^{\f}$ when the ghost vertex is isolated.

We begin with the free case.
For a spanning tree $T$ of $G$ and an edge $e \in E(G) \setminus T$, let $C^T_e$ be the unique cycle in $T \cup \{e\}$ (which necessarily contains~$e$).

\begin{lemma}\label{lem:gen_set_free}
  Suppose that $v^*$ is isolated in $G^*$ and let $T$ be an end faithful spanning tree of $G$ (or of each connected component if $G$ is not connected).
  Then the collection $\{C^T_e\}_{e \in E(G) \setminus T}$ is a locally finite generating set for $\cE^{\f}$.
\end{lemma}

\begin{proof}
  We first show that $\cC := \{C^T_e\}_{e \in E(G) \setminus T}$ is a locally finite.
  An edge $e\not\in T$ is contained in a unique cycle $C^T_e$.
  Suppose that some edge $e\in T$ is contained in infinitely many cycles $C^T_{e_j}$.
  Taking a subsequential limit of $C_{e_j}$ yields a bi-infinite path in $T$, which has infinitely many edges (namely the $e_j$) between its two ends.
  This means that there are two ends of the tree, which are in the same end of $G$, which is a contradiction to end faithfulness of $T$.

  Let us next show that $\cC$ generates $\cE^{\f}$. To this end, it is enough to show every finite cycle $H$ in $G$ is generated by $\cC$. Toward showing this, fix $H$ and note that
  \[ H' := H \oplus \sum_{e \in H \setminus T} C_e^T , \]
 is a finite even subgraph which is contained in $T$. Since a tree contains no non-trivial finite even subgraphs, $H'$ must be empty, which means that $H$ is generated by $\cC$.
\end{proof}


We next address the wired case, for which we require some definitions.
Let $F$ be a one-ended spanning forest of $G$ (i.e., a subgraph of $G$ containing no cycle and whose connected components are all infinite and one ended).
For an edge $e \in E(G) \setminus F$, let $C^F_e$ be the unique cycle or bi-infinite path in $F \cup \{e\}$ (which necessarily contains~$e$).
For an edge $e \in E(G^*) \setminus E(G)$, let $C^F_e$ be the unique ray starting from $v^*$ in $F \cup \{e\}$ (which first crosses $e$ and then goes along the unique ray in $F$ starting at the other endpoint of~$e$).

We say an end $\mathsf{e}$ of a subgraph $H$ of $G$ is \textbf{dense} in $G^*$ if there exists a ray $p \in \mathsf e$ such that infinitely many vertices in $p$ are adjacent to $v^*$.
We say that $G^*$ is \textbf{end dense} if every end of $G$ is dense in $G^*$. 

Let $T$ be a $k$-ended spanning tree of $G$ whose every end is dense in $G^*$.
For every $e \in E(G^*) \setminus E(G)$, consider the $k$ rays starting from $v^*$ in $T \cup \{e\}$ (which first cross $e$ and then go along a ray of $T$ (one for each end) from the other endpoint of $e$). Suppose that ray $i$ is $(v^*, u_1,u_2,\ldots)$ and let $u_n$ be the first vertex which is adjacent to $v^*$ (which exists since the ends of $T$ are dense in $G^*$). Let $C^T_{e,i}$ be the cycle $(v^*,u_1, \ldots, u_k, v^*)$.
%

\begin{lemma}\label{lem:gen_set_forest}~

\begin{itemize}
\item Let $F$ be a one-ended spanning forest of $G$. Then the collection $\{C^F_e\}_{e \in E(G^*) \setminus F}$ is a locally finite generating set for $\cE^{\w}$.

\item Suppose that $G$ is connected and $k$-ended with $1 \le k<\infty$ and let $T$ be an end faithful spanning tree of $G$ whose every end is dense in $G^*$. Then the collection $\{C^T_e\}_{e \in E(G) \setminus T} \cup \{C^{T}_{e,j}\}_{e \in E(G^*)\setminus E(G), 1\le j \le k}$ is a locally finite generating set for $\cE^{\w}$.
\end{itemize}
\end{lemma}
\begin{proof}
For the first item, denote $C_e := C^F_e$ and $\cC:=\{C_e\}_{e \in E(G^*) \setminus F}$.
Let us first prove that $\cC$ is locally finite. Take an edge $f \in E(G^*)$. If $f \notin F$ then the only element of $\cC$ which contains $f$ is $C_f$. Now assume that $f \in F$.
Let $P$ be the set of vertices $v \in V(G)$ which belongs to the unique finite component of $F \setminus \{e\}$. Let $\partial_F P$ denote the set of edges $e \in E(G^*) \setminus F$ incident to $P$. Then $\partial_F P$ is also finite (since the degrees of vertices in $V(G)$ are finite). It is not hard to see that any element of $\cC$ containing $f$ must also intersect $\partial_F P$. As we have already noted, for each $e \notin F$ there is only one element of $\cC$ containing $e$. In particular, the number of elements of $\cC$ containing $f$ is at most $|\partial_F P|$.


Let us next show that $\cC$ generates $\cE^{\w}$. To this end, let $H \in \cE^{\w}$ and define $$H' := \sum_{e \in H \setminus F} C_e .$$
Note that this is well defined since $\cC$ is locally finite. Note that $H'':=H \oplus H'$ is an even subgraph which is contained in $F$.
However, $F$ is a one-ended forest in $G$, and hence contains no nonempty even subgraphs (observe that every nonempty even graph contains either a cycle or a bi-infinite path). Thus, $H''$ is empty, which means that $H=H'$ is generated by $\cC$.

We now turn to the second item. By \cref{lem:gen_set_free}, all finite cycles not involving the ghost vertex are generated. Also any ray started at $v^*$ and then going to infinity along $T$ is also generated (add the finite cycles involving $v^*$ along the ray, using the fact that the tree is end dense). Now take a cycle containing the ghost vertex and without loss of generality assume the cycle is simple. Adding it to the two infinite rays in the generating set containing the edges incident to the ghost vertex in the cycle, we obtain a bi-infinite path in $G$. So it is enough to generate any bi-infinite path $p$ in $G$. If the two ends of $p$ belong to the same end of $G$, then $p$ can be generated by the finite cycles formed by the infinitely many paths joining them and we are done. So assume that the two ends of $p$ belongs to two different ends. To generate $p$ do the following.  Take any two rays started at $v^*$ and going to infinity along the corresponding two ends of $T$ (using the fact that $T$ is end faithful), and add them. This produces an even subgraph $p'$ with two ends, one for each end of $p$. Since these paths can be joined together by infinitely many finite paths (using the definition of ends), it is easy to see that $p$ can be generated by adding finite cycles to $p'$.

Now we show local finiteness of the generating set in the second item. As in the proof of \cref{lem:gen_set_free}, every edge not in $T$ is in exactly one cycle and every edge in $T$ is in only finitely many cycles of $\{C^T_e\}_{e \in E(G) \setminus T}$. It remains to show that every edge $h$ in $T$ contains only finitely many cycles of $\{C^{T}_{e,j}\}_{e \in E(G^*)\setminus E(G), 1\le j \le k}$. Note that there are $k$ rays in $T\setminus \{h\}$ starting from the endpoints of $h$. Take the first vertex in each ray which is connected to the ghost vertex, and join those vertices to $h$ by paths in $T$. Denote by $U$ the subgraph formed by the union of these paths and $h$. Notice that $T \setminus U$ divides $T$ into some finite and some (at most $k$) infinite components. It is now straightforward to see that the vertex set of a cycle in $\{C^T_{e,j}\}_{e \in E(G^*)\setminus E(G), 1\le j \le k}$ which passes through $h$ is contained in $U \cup \{v^*\}$. We conclude that the collection is locally finite. This completes the proof of the second item.
\end{proof}
\begin{remark}
We will use the wired uniform spanning forest as our candidate of $F$ (first item of \cref{lem:gen_set_forest}) and we will show in \cref{USFfactor} that such a forest can be obtained as a factor. The second item of \cref{lem:gen_set_forest} will be used only in the case that $k=2$. 
We do not know how to obtain a candidate for $T$ of \cref{lem:gen_set_free} as a factor of i.i.d.\ in the nonamenable, infinitely ended setting. In fact, we will only use \cref{lem:gen_set_forest} in the case that $G$ is one ended.
\end{remark}

We now describe two sufficient conditions under which $\cE^{\f}$ and $\cE^{\w}$ are equal.  

\begin{lemma}\label{lem:end_dense}
If $G^*$ is end dense, then it is one-ended and $\cE^{\w} = \cE^{\f}$.
\end{lemma}
\begin{proof}
To show that $\cE^{\w} = \cE^{\f}$, it is enough to show that any bi-infinite path in $G^*$ or ray starting at $v^*$ in $G^*$ can be generated by finite cycles in $G^*$.
Let $p$ be a ray in $G^*$ starting at $v^*$. Each additional visit of $p$ to $v^*$ creates a finite cycle. By adding these finite cycles to $p$, we may assume that $p$ does not visit $v^*$ after its initial vertex, so that $p$ is contained in $G$ except for its first edge.
Let $\mathsf e$ be the end of $G$ corresponding to $p$. Since $G$ is end dense, we can find a ray $q \in \mathsf e$ in $G$ which contains infinitely neighbors of $v^*$. Without loss of generality, assume that the first vertex of $q$ is $v^*$. Note that $q$ and $p$ have infinitely many disjoint paths between them (since they belong to the same end). Using these paths, it is easy to see that we can find a collection of finite cycles $\Gamma$ such that $p=q\oplus \sum_{\gamma \in \Gamma}\gamma$. Now observe that $q$ can be obtained as a sum of finite cycles in $G^*$ using the edges connecting it to $v^*$. This shows that $p$ is the sum of finite cycles. To handle the case when $p$ is a bi-infinite path in $G^*$, take any finite path $q$ from $v^*$ to $p$, and note that $p$ is the sum of two rays starting from $v^*$, both of which initially go along $q$ and then split and go along either end of $p$. Since the latter rays are generated by finite cycles, so is $p$.
This completes the proof that $\cE^{\w} = \cE^{\f}$.

Finally, we show $G^*$ is one-ended. Any two rays in $G^*$ containing infinitely many neighbors of $v^*$ are clearly equivalent. Suppose that some ray $p$ in $G^*$ has only finitely neighbors of $v^*$. Clearly, this ray visits $v^*$ only finitely often, and hence belongs to some end of $G$. Since $G^*$ is end dense, we can find a ray in $G$ which is equivalent to $p$ and contains infinitely neighbors of $v^*$. This shows that any two rays are equivalent, thereby completing the proof.
\end{proof}

\begin{lemma}\label{cor:1end}
If each connected component of $G$ is one-ended and $v^*$ is isolated, then $\cE^{\w} = \cE^{\f}$.
\end{lemma}
\begin{proof}
Since $v^*$ is isolated, we only need to show that every bi-infinite path $p$ in $G$ belongs to $\cE^{\f}$. Since each connected component of $G$ is one-ended, there are infinitely many disjoint paths between the two ends of $p$. By summing the finite cycles created by these paths, we obtain $p$.
\end{proof}
\begin{remark}
If $G$ is one-ended and $v^*$ has infinite degree, then \cref{lem:end_dense} implies equality of the free and wired cycle spaces. However, if $G$ is one-ended and $v^*$ has finite non-zero degree, then $\cE^{\w} \neq \cE^{\f}$: there is no way to generate a ray starting from $v^*$ using only finite cycles.
\end{remark}

\subsection{Uniform even subgraphs of infinite graphs} \label{sec:proj}

We now turn to the question of defining and generating uniform even subgraphs of infinite graphs.
Let $G$ be a locally finite infinite graph and let $G^*$ be a graph obtained by connecting a subset of vertices of $G$ to a ghost vertex $v^*$. We emphasize that $G$ and $G^*$ are not assumed to be connected here (and each may have infinitely many finite or infinite connected components).

We aim to give a constructive definition of the free and wired uniform even subgraphs on $G^*$. As usual, this can be done via an exhaustion of $G^*$ by finite subgraphs of $G^*$. While this will be sufficient for our application to the Loop $O(1)$ model in \cref{sec:loop} in the free case (where the exhaustion will be by samples of the free FK-Ising on finite graphs, which increase as the finite graphs increase), our application in the wired case requires working within a larger ambient graph $\tilde G$ containing $G^*$ (since the wired FK-Ising decreases as the finite graph increases, when the configuration in the finite graph is viewed as a configuration on the infinite graph by assigning 1s to all edges outside the finite graph).


We begin by addressing the free case.
Let $\{G_n\}_{n \ge 1}$ be an increasing sequence of finite subgraphs of $G^*$ which exhaust $G^*$.
Let $\cE_n^{\f}$ denote the cycle space of $G_n$, and let $\cE^{\f}$ denote the free cycle space of $G^*$. Let $U_n^{\f}$ denote a uniformly chosen element in $\cE_n^{\f}$. We will see that $U_n^{\f}$ converges in distribution to an element $U^{\f}$ of $\cE^{\f}$ (whose law does not depend on the chosen exhaustion), which we will call the \textbf{free uniform even subgraph} of $G^*$.

Let us now consider the wired case. We could simply take an exhaustion $\{G_n\}$ of $G^*$ as in the free case, but as we have mentioned, we need to allow for a more general situation. Let $\tilde G$ be the graph obtained from $G$ by connecting \emph{all} vertices to the ghost vertex $v^*$ (so that $G^*$ is a spanning subgraph of $\tilde G$). Let $\{\tilde G_n\}_{n \ge 1}$ be an increasing sequence of finite subgraphs of $\tilde G$ which exhaust $\tilde G$. Let $\tilde G_n^{\w}$ be the graph obtained from $\tilde G$ by identifying all the vertices not in $\tilde G_n$ with the ghost vertex $v^*$ and then removing the self loops from $v^*$ (see \cref{rmk:ghost_boundary} for some subtleties related to this definition where the importance of identifying the boundary with the ghost vertex is discussed). Note again that we allow multiple edges in $\tilde G_n^{\w}$. Note also that the edge set of $\tilde G_n^{\w}$ can be seen as a subset of the edge set of $\tilde G$. We often also identify $\tilde G_n^{\w}$ with its edge set $E(\tilde G_n^{\w}) \subset E(\tilde G)$.
Let $G_n^{\w}$ be a spanning subgraph of $\tilde G_n^{\w}$ such that the sequence $\{E(G_n^{\w}) \cup (E(\tilde G) \setminus E(\tilde G_n^{\w}))\}_{n \ge 1}$ decreases to $G^*$. Note that one potential situation like this is when $G^* = \tilde G$ and $G_n^{\w}=\tilde G_n^{\w}$. We remark that the only role of $\tilde G$ and $\tilde G_n$ is to give a proper sense in which $G^{\w}_n$ decreases to $G^*$.\footnote{In our applications $G^{\w}_n$ will be a sample from the wired FK-Ising model on $\tilde G_n^{\w}$.}
Let $\cE_n^{\w}$ denote the cycle space of $G_n^{\w}$, and let $\cE^{\w}$ denote the wired cycle space of $G^*$. Let $U_n^{\w}$ denote a uniformly chosen element in $\cE_n^{\w}$. We will see that $U_n^{\w}$ converges in distribution to an element $U^{\w}$ of $\cE^{\w}$ (whose law does not depend on the choices of $\tilde G$, $\tilde G_n$ and $G_n^{\w}$), which we will call the \textbf{wired uniform even subgraph} of $G^*$.

Our goal is twofold: to show that the limits of $U_n^{\f}$ and $U_n^{\w}$ exist and to find a useful description of their limits $U^{\f}$ and $U^{\w}$. 
We start with a deterministic lemma which states that the projection of the finite cycle space on a fixed set converges to the projection of the infinite cycle space in a suitable sense.
For any $\omega \subset \{0,1\}^{E(\tilde G)}$ and $E \subset E(\tilde G)$, define the \textbf{projection} of $\omega$ onto $E$ to be $\{\omega_e\}_{e \in E} \in \{0,1\}^E$. For any $\cA \subset \{0,1\}^{E(\tilde G)}$, let $\mathsf{proj}(\cA, E)$ be the set of projections of all $\omega \in \cA$ onto $E$.
For $n\ge k$, denote
\[ \cE_{n,k}^{\f} := \mathsf{proj}(\cE_n^{\f},G_k) \qquad\text{and}\qquad \cE_{n,k}^{\w} := \mathsf{proj}(\cE_n^{\w},G_k^{\w}) .\]
Similarly, denote
\[ \cE_{\infty,k}^{\f} := \mathsf{proj}(\cE^{\f},G_k) \qquad\text{and}\qquad \cE_{\infty,k}^{\w} := \mathsf{proj}(\cE^{\w},G_k^{\w}) .\]

\begin{lemma}\label{lem:proj}
For any $k$ there exists an $N$ such that for all $n >N$, $$\cE^{\f}_{n,k} = \cE^{\f}_{\infty, k} \qquad\text{and}\qquad \cE^{\w}_{n,k} = \cE^{\w}_{\infty, k}.$$
\end{lemma}
\begin{proof}
Let us first tackle the free case. Since any element in $\cE^{\f}_{n}$ is a finite sum of finite cycles in $G_n$, it must be in $\cE^{\f}$ since $G_n$ is a subgraph of $G^*$. Consequently, its projection to $G_k$ is in $\cE^{\f}_{\infty, k}$, that is, $\cE^{\f}_{n,k} \subseteq \cE^{\f}_{\infty,k}$ for all $ n \ge k \ge 1$. Conversely, take an element $H$ in $\cE^{\f}_{\infty,k}$ and suppose $H = \mathsf{proj} (\tilde H, G_k)$ for some $\tilde H \in \cE^{\f}$. Let $\tilde H_m \to \tilde H$ where $\tilde H_m$ is a finite sum of finite cycles in $G^*$. This implies that $\mathsf{proj}(\tilde H_m, G_k)$ is the same for all $m \ge m_0(H)$. Since $H_{m_0}$ is a sum of finitely many finite cycles, it must be in $\cE^{\f}_{n,k}$ for some $n \ge N(H)$. Since $\cE^{\f}_{\infty, k}$ contains only finitely many elements, we can choose an $N=N(k)$ such that $\cE^{\f}_{\infty,k} \subseteq \cE^{\f}_{n,k}$ for all $n \ge N$.  

Now let us turn to the wired case, which is a bit more involved. We rely on the monotonicity of $G_n^{\w}$ which tells us that
\[ E(G^*) \cap  E(\tilde G_n^{\w}) \subset E(G_n^{\w}) \qquad\text{for any }n .\]
It is easy to see that the projection of any finite cycle or bi-infinite path in $G^*$ onto $G_n^{\w}$ is in $\cE^{\w}_n$. Indeed, when such a cycle or path enters or leaves $\tilde G_n^{\w}$, it does so through the wired vertex, and these edges are all in $G_n^{\w}$ by the above displayed. Thus the projection results in a union of finitely many cycles in $G_n^{\w}$ which is in $\cE_n^{\w}$. The same argument works for a ray starting at $v^*$.
Thus, $\mathsf{proj}(\cE^{\w},G_n^{\w}) \subset \cE_n^{\w}$, and hence $\cE^{\w}_{\infty,k} \subseteq \cE^{\w}_{n,k}$ for all $n \ge k \ge 1$.

Now we prove the reverse direction, i.e., $\cE^{\w}_{\infty,k} \supseteq \cE^{\w}_{n,k}$ for all $n \ge N(k)$. Assume the converse: $\cE^{\w}_{n,k} \setminus \cE^{\w}_{\infty, k} \neq \emptyset$ for infinitely many $n$. By compactness, there is some $B$ which belongs to $\cE^{\w}_{n,k} \setminus \cE^{\w}_{\infty, k}$ for infinitely many $n$. Choose integers $n_i\to\infty$ and elements $H_i \in \cE^{\w}_{n_i}$ such that $\mathsf{proj}(H_i,G_k^{\w})=B$. By compactness, $H_i$ has a subsequential limit $H$. Clearly, $\mathsf{proj}(H,G_k^{\w})=B$ and $H \subset G^*$ since $G^{\w}_n \to G^*$. Thus, it suffices to show that $H \in \cE^{\w}$, as this will lead to a contradiction to the fact that $B \notin \cE^{\w}_{\infty,k}$. To this end, we fix $v \neq v^*$ and show that the degree of $v$ in $H$ is even. Indeed, $H_i$ is an even subgraph of $G_{n_i}^{\w}$, so the degree of $v$ in it is even for all $i$. Since $v$ has finite degree in $\tilde G$, it follows that its degree is even also in the limit $H$. This completes the proof. 
\end{proof}

Let $\cC^{\f} = (C^{\f}_i)_{i \ge 1}$ be a locally finite generating set for $\cE^{\f}$ consisting of only finite sets $C^{\f}_i$, and let $\cC^{\w} = (C^{\w}_i)_{i \ge 1}$ be any locally finite generating set for $\cE^{\w}$. Such generating sets exist by \cref{lem:gen_set_GJ}. Let $\{\ve_i\}_{i \ge 1}$ be i.i.d.\ Bernoulli$(1/2)$ random variables and define 
\begin{equation}
U^{\f} := \sum_{i \ge 1}\ve_i C^{\f}_i \qquad  \text{and} \qquad U^{\w} = \sum_{i \ge 1}\ve_i C^{\w}_i .\label{eq:U}
\end{equation}
Note that these sums exist by local finiteness. Recall that $U_n^{\f}$ (resp.\ $U_n^{\w}$) denotes a uniform element in $\cE_n^{\f}$ (resp.\ $\cE_n^{\w}$).
\begin{prop}[free and wired uniform even subgraphs]\label{prop:limit_projection}
There exists a coupling between $\{U_n^{\f}\}_{n \ge 1}$  and $U^{\f}$ such that for every $k$, there exists an $N$ such that $\proj(U_n^{\f}, G_k) = \proj (U^{\f},G_k)$ for all $n \ge N$ almost surely. In particular, $U^{\f}_n \to U^{\f}$ almost surely and the law of $U^{\f}$ does not depend on the generating set used in its definition. The same holds for $(G^{\w}_n,U_n^{\w},U^{\w})$ in place of $(G_n,U_n^{\f},U^{\f})$.
\end{prop}
\begin{proof}
Let us first prove the free case. Let $\cC_k$ be the elements in $\cC^{\f}$ which intersect $G_k$. Given $n$, let $k(n)$ be the largest $k$ such that every element in $\cC_k$ lies completely inside $G_n$.
We claim that $\cC_{k(n)}$ can be extended to a generating set $\cC_{k(n)} \cup \{C_{n,j}\}_{1 \le j \le \ell_n}$ for $\cE^{\f}_n$ so that the added elements $\{C_{n,j}\}_j$ do not intersect $G_{k(n)}$. Let us first explain how to deduce the proposition from this. Let $\{\eta_j\}_{j \ge 1}$ be i.i.d.\ Bernoulli$(1/2)$ random variables, independent of $\{\ve_i\}_i$. Define
\[ U^{\f}_{n} := \sum_{i : C^{\f}_i \in \cC_{k(n)}} \ve_i C^{\f}_i + \sum_{j=1}^{\ell_n} \eta_j C_{n,j} .\]
By \cref{prop:cycle_finite}, $U^{\f}_{n}$ is a uniform element in $\cE^{\f}_n$. Also $\mathsf{proj}(U_n, G_{k(n)}) = \mathsf{proj}(U, G_{k(n)})$ almost surely. Noting that $k(n)\to\infty$ as $n\to\infty$ (since $\cC_k$ consists of finitely many finite sets), the proposition follows.

We now prove the existence of $\{C_{n,j}\}_j$ as above. Fix $n$ and set $k:=k(n)$.
Given $H \in \cE^{\f}_{n}$, let $\psi(H)$ be any element in $\mathsf{span}(\cC_k) \subset \cE^{\f}_{n}$ with $\proj(\psi(H), G_k) = \proj(H, G_k)$. Such an $\psi(H)$ exists since $\cC^{\f}$ is a locally finite generating set for $\cE^{\f}$, and $H \in \cE^{\f}_n \subset \cE^{\f}$ (take a set of elements in $\cC^{\f}$ which sums to $H$, and keep only those which intersect $G_k$; alternatively, one could use \cref{lem:proj} and increase $n$ appropriately). Now define $$\cD := \{H \oplus \psi(H): H \in \cE^{\f}_{n}\}.$$
Notice that the elements of $\cD$ are all disjoint from $G_k$. Also, $\cC_k \cup \cD$ generates $\cE^{\f}_n$ as any $H \in \cE^{\f}_n$ can be written as $\psi(H) \oplus (H\oplus H')$ and $\psi(H) \in \mathsf{span}(\cC_k)$ and $H\oplus \psi(H) \in \cD$. This proves the claim with $\{C_{n,j}\}_j = \cD$.

The proof for the wired case is very similar.
Let $\cC_k$ be the elements in $\cC^{\w}$ which intersect $G^{\w}_k$.
Given $n$, let $k(n)$ be the largest $k$ such that $\cE^{\w}_{n,k} = \cE^{\w}_{\infty, k}$. Note that $k(n)\to\infty$ as $n\to\infty$ by \cref{lem:proj}. It suffices to show that $\mathsf{proj}(\cC_{k(n)},G_n^{\w})$ can be extended to a generating set $\cC_{k(n)} \cup \{C_{n,j}\}_{1 \le j \le \ell_n}$ for $\cE^{\w}_n$ so that the added elements $\{C_{n,j}\}_j$ do not intersect $G^{\w}_{k(n)}$, as a desired coupling can then be defined as in the free case. This follows as in the free case (using the trivial fact that projections and summations commute).
\end{proof}

In light of \cref{prop:limit_projection}, in order to prove that $U^{\f}$ or $U^{\w}$ is a factor of i.i.d., we need to construct a locally finite generating set for $\cE^{\f}$ or $\cE^{\w}$ in an invariant manner (as a factor of i.i.d.).  We now record the distributional versions of  \cref{lem:end_dense,cor:1end}. Recall the definition of end dense from before \cref{lem:end_dense}.
\begin{lemma}\label{lem:wired=free}
If $G^*$ is end dense, then $U^{\w} = U^{\f}$ in law. If each component of $G$ is one-ended and $v^*$ is isolated in $G^*$, then $U^{\w} = U^{\f}$ in law.
\end{lemma}
\begin{proof}
It follows immediately from \cref{lem:end_dense,cor:1end} that in either of these cases, $\cE^{\w} = \cE^{\f}$.  Hence we are done using \cref{prop:limit_projection}.
\end{proof}

\begin{remark}\label{rmk:ghost_boundary}
We finish this section with a remark about the boundary condition for the wired case. Recall that in the definition of $G_n^{\w}$, we identified the boundary vertex with the ghost vertex. One might also define a boundary condition where this is not the case (the wired vertex and ghost vertex are distinct). In this setting, the limit of the wired even subgraphs include only those infinite rays in the generating set which belong to a dense end. We do not pursue these avenues in more details. 
 For our applications, the graphs in question will be FK-Ising clusters which are end dense for any positive external field, and have no ghost vertex if there is no external field, for which these two definitions coincide. 
 \end{remark}

\subsection{Loop $O(1)$ model on infinite graphs}\label{sec:loop}
Fix a locally finite infinite connected graph $G = (V,E)$ and let $\tilde G$ be the graph where {\emph {all}} the vertices in $G$ are connected to a ghost vertex $v^*$ as defined in the previous section.
Recall that configuration in the Loop $O(1)$ and FK-Ising model are elements of $\{0,1\}^{E \cup V}$. There is an obvious identification between such elements and spanning subgraphs of $\tilde G$: we identify $\omega \in \{0,1\}^{E \cup V}$ with the spanning subgraph $\omega^*$ whose edge set is $\{e \in E: \omega_e =1\} \cup \{\{v,v^*\}: v \in V, \omega_v =1\}$. Note that for the wired measures, $\bP_{G_n^{\w},x,y}^{\w}$ and $\phi_{G_n^{\w},p,p_h}^{\w}$, we further identify the ghost vertex with the boundary vertex $\Delta$ (see \cref{rmk:ghost}).

In particular, given $\omega \in \{0,1\}^{E \cup V}$, this allows us to talk about the wired and free uniform even subgraphs of $\omega$ (where $\omega^*$ plays the role of $G^*$ in \cref{sec:proj}).
Note also that an even percolation configuration $\eta \in \{0,1\}^{E \cup V}$ (recall the definition from \cref{sec:introduction}) precisely corresponds to an element $\eta^*$ in the wired cycle space $\cE^{\w}$ of $\tilde G$.

We present a corollary of \cref{prop:limit_projection} (together with \cref{prop:coupling}) in the context of the Loop $O(1)$ model. Let $G_n$ be an increasing sequence of induced finite subgraphs which exhaust $G$. Let $G_n^{\w}$ denote the graph obtained by identifying all the vertices outside $G_n$ into a single vertex and removing all the self loops (as before, this graph may have multiple edges). Recall the free and wired FK-Ising measures, $\phi^{\f}_{G,p,p_h}$ and $\phi^{\w}_{G,p,p_h}$, from \cref{sec:FK_ising}. 

\begin{prop}\label{prop:loop_a.s._conv}
Let $\beta,h,x,y,p,p_h$ be such that $x = \tanh(\beta)$, $y = \tanh(h)$, $p = 1-e^{-2\beta}$ and $p_h = 1-e^{-2h}$.
Let $\omega^{\f}$ be sampled from $\phi^{\f}_{G,p,p_h}$ and let $\eta^{\f}$ be a sample of the free uniform even subgraph of $\omega^{\f}$. Then $\bP^{\f}_{G_n,x,y}$ converges to $\bP^{\f}_{G,x,y}$, which is the distribution of $\eta^{\f}$.
A similar statement holds for the wired case.
\end{prop}

\begin{proof}

It is a standard fact using monotonicity that the probability measures $\phi^{}_{G_n,p,p_h}$ stochastically increase in $n$ and converge to a limiting measure $\phi^{\f}_{G,p,p_h}$ called the \emph{free} FK-Ising model on $G$. Similarly, $\phi_{G^{\w}_n,p,p_h}^{\w}$ is a stochastically decreasing sequence which converges to the \emph{wired} FK-Ising measure $\phi^{\w}_{G,p,p_h}$ on $G$. Also, for each $n$, the wired (resp.\ free) measure is the largest (resp.\ smallest) possible FK-Ising measure in the sense of stochastic domination.
Viewing $\phi^{}_{G_n,p,p_h}$ (resp.\ $\phi_{G^{\w}_n,p,p_h}^{\w}$) as a probability measure on $\{0,1\}^{E\cup V}$ which is deterministically 0 (resp.\ 1) on all vertices and edges outside of $G_n$, we obtain a coupling between $(\omega^{\f},\omega^{\w},(\omega_n^{\f},\omega_n^{\w})_{n \ge 1})$, where $\omega^{\f}_n \sim \phi^{}_{G_n,p,p_h} $, $\omega_n^{\w} \sim \phi^{\w}_{G_n,p,p_h}$,  $\omega^{\f}\sim \phi^{\f}_{G,p,p_h} $ and $\omega^{\w} \sim \phi^{\w}_{G_n,p,p_h}$ under which $\omega^{\f}_n \le \omega^{\f}_{n+1} \le  \omega^{\f} \le \omega^{\w} \le \omega_{n+1}^{\w} \le \omega_n^{\w}$ (in the pointwise partial order) for all $n$ almost surely (see the proof of \cite[Theorem 4.91]{grimmett2006random} for details on why such a coupling exists).

Let us now turn to the Loop $O(1)$ model. Consider first the free case.
Let $\eta_n$ be a uniform even subgraph of $\omega^{\f}_n$ conditionally on $((\omega^{\f}_n)_n,\omega^{\f})$.
By \cref{prop:coupling}, the (unconditional) distribution of $\eta_n$ is $\bP^{\f}_{G_n,x,y}$.
By \cref{prop:limit_projection}, conditionally on $((\omega^{\f}_n)_n,\omega^{\f})$, $\eta_n$ converges in distribution to $\eta^{\f}$ (using $G_n=\omega^{\f,*}_n$ and $G^*=\omega^{\f,*}$). In particular, $\eta_n$ converges in distribution to $\eta^{\f}$. This completes the proof for the free case.

We now turn to the wired case.
Let $\eta_n^{\w}$ be a uniform even subgraph of $\omega^{\w}_n$.
By \cref{prop:coupling} (see \cref{rmk:ghost}), the (unconditional) distribution of $\eta^{\w}_n$ is $\bP^{\w}_{G_n^{\w},x,y}$ conditionally on $((\omega^{\w}_n)_n,\omega^{\w})$.
By \cref{prop:limit_projection}, conditionally on $((\omega^{\w}_n)_n,\omega^{\w})$, $\eta^{\w}_n$ converges in distribution to $\eta^{\w}$ (using $G^{\w}_n=\omega^{\w,*}_n$ and $G^*=\omega^{\w,*}$ and that $\omega^{\w}_n$ decreases to $\omega^{\w}$). In particular, $\eta^{\w}_n$ converges in distribution to $\eta^{\w}$. This completes the proof for the wired case.
%
\end{proof}

\begin{remark}
It was previously known that the free and wired Loop $O(1)$ model converges in law (see \cite[Theorem 2.3]{aizenman2015random}). However the proof goes through the convergence of correlation functions in the gradient Ising model, which is only related to the Loop $O(1)$ model in the distributional sense. An inspection of the proof of \cref{prop:loop_a.s._conv} provides an alternative proof of this convergence via a reasonably explicit coupling. In fact, once we sample the FK-Ising, the coupling of \cref{prop:limit_projection} shows that the projection of the Loop $O(1)$ on a fixed set of edges stabilizes for $n >N$ for a choice of $N$ depending only on the FK-Ising sample.
\end{remark}

\begin{lemma}\label{lem:loopO1end_dense}
Let $p \in [0,1]$ and $p_h \in (0,1]$. Let $\omega^{\f} \sim \phi^{\f}_{G,p,p_h}$ and $\omega^{\w} \sim \phi^{\w}_{G,p,p_h} $. Then both $\omega^{\f,*}$ and $\omega^{\w, *}$ are almost surely end dense.
\end{lemma}
\begin{proof}
The proof for the free and the wired are the same, so we prove just the free case.
By finite energy and Holley's criterion (\cite[Theorem 2.3 (b)]{grimmett2006random}), conditionally on $\omega^* \cap E(G)$, we have that $\{\omega^*_{\{v,v^*\}}\}_{v \in V(G)}$ dominates a Bernoulli percolation with parameter $p_h/(2-p_h)$, from which the result easily follows.\end{proof}
We remark that \cref{lem:loopO1end_dense} will only be used in the free case.
\begin{lemma}\label{lem:tree_end_dense}
Let $G$ be a connected $k$-ended graph for some positive integer $k $ and let $T$ be a spanning forest of $G$ with each tree in it being infinite with finitely many ends. Let $p_h \in (0,1]$, $p \in [0,1]$ and $\omega^{\f} \sim \phi^{\f}_{G,p,p_h}$ and $\omega^{\w} \sim \phi^{\w}_{G,p,p_h} $. Then $T^*$ is almost surely end dense.
\end{lemma}
\begin{proof}
This proof is the same as in \cref{lem:loopO1end_dense}. We leave the details to the reader.
\end{proof}

\section{Wired uniform spanning forest as a factor of i.i.d.}\label{sec:USFfactor}

In this section, we prove \cref{thm:UST} (restated below as \cref{USFfactor}).
Thus, we assume that $(G, \rho)$ is a transient, connected, random
rooted graph and aim to construct the wired uniform spanning forest (WUSF) on it as a graph factor of i.i.d.
The WUSF is a distributional limit of the wired (rooted) spanning tree $(T_n, \rho)$ of
any exhaustion of $(G, \rho)$ by finite subgraphs $(G_n,\rho)$.
This yields the triplet $(G,\rho,T)$, called the WUSF on $(G,\rho)$, which is a unimodular, random rooted marked graph. 

Given a transient $(G, \rho)$, the WUSF can be generated by
the celebrated Wilson's algorithm rooted at infinity.
We briefly describe this procedure (see \cite[Proposition
10.1]{lyons2017probability} for details).
A loop erased random walk is a random walk whose loops are erased
chronologically (we refer to \cite[Section 4.1]{lyons2017probability}
for a more precise definition).
Condition on $(G, \rho)$ and arbitrarily order the vertices of $G$.
Sample an independent random walks from each vertex of $G$ one by one in
order.
The loop erasure of the first walk will generate an infinite path
because $G$ is transient.
Run the second walk until it hits the first path or escapes to infinity,
and loop erase it.
Iterate this process until a path from every vertex is sampled.
It is known (see e.g.\ \cite{AHNR2}) that this creates a unimodular,
random rooted network $(G,\rho,T)$ where $T$ is
distributed as the WUSF on $G$.
In particular, the law of the resulting tree does not depend on the
order chosen for the vertices (though the resulting realization does).
Finally, we also rely on the related cycle popping algorithm, as
described below.

Each vertex has an arrow pointing uniformly at random to one of its
neighbors, except for one vertex which is fixed to be the sink.
If the resulting directed graph contains an oriented cycle, ``pop'' it
by choosing new independent arrows at each vertex of the cycle.
Repeat until no directed cycle exists.
This results in a uniform spanning tree with each edge oriented towards
the sink.
The key insight is that the resulting tree does not depend on the order
of the cycles chosen (\cite[Lemma~4.2]{lyons2017probability}). Our proof will rely on an extension of this algorithm to an infinite graph setting.


The reader should recall the definition of a graph factor of i.i.d.\
from \cref{sec:coding_def}.

\begin{thm}\label{USFfactor}
   Let $(G, \rho)$ be a connected, transient, random rooted graph and
conditioned on it let  $T$ be a  WUSF on $G$. Then $(G, \rho, T)$ is a
graph factor of i.i.d.
\end{thm}

\begin{proof}
Throughout, we condition on $(G, \rho)$.
The core idea is to extend the cycle popping algorithm of Wilson to the
infinite setting.
For each vertex $x$, let $S_x = (S_x^1,S_x^2,\dots)$ be a sequence of
oriented edges, where each $S_x^i$ is uniformly and independently chosen
from all edges emanating from $x$.
(There is no `sink' vertex and the root $\rho$ plays no special role,
the role of the sink vertex is played by infinity.
Indeed, rooting at infinity is required to get a graph factor in the end.)
We refer to $S_x$ as the \textbf{stack} at $x$ (thinking of $S_x^1$ as
the top of the stack), to the $S_x^i$ as \textbf{arrows}, and to $i$ as the
\textbf{color} of the arrow $S_x^i$.
We will construct $(G, \rho, T)$ as a graph factor of $(G, \rho, \Xi)$,
where $\Xi = (S_x)_{x \in V}$ are i.i.d.\ variables.

The construction proceeds as follows.
First, reveal the top arrow at every vertex.
This creates a directed graph with oriented paths and cycles, and
potentially infinite or bi-infinite oriented paths.

As in the finite case, we repeatedly choose a finite cycle (if one
exists) and pop it --- throw away the top arrow at each stack along the
cycle so as to expose the next arrows there.
In contrast to the finite setting, the top arrows may (and a.s.\ will)
create infinitely many finite cycles, so we will need to be careful with
how we choose which cycles to pop.
We will keep track of the cycles popped at any given moment, including
the information of the colors of the arrows in the cycle (we refer to
these as \textbf{colored cycles}).
After a certain number of poppings, the collection of topmost arrows in
the stacks are called \textbf{exposed arrows}.
Note that popping a cycle increments the colors of the exposed arrows at
the vertices of the cycles.

We now explain how to pop the cycles in such a way so that after all the
popping has finished, we are left with exposed arrows which describe the
WUSF as a factor of $\Xi$.
In an infinite graph, it is possible that after infinitely many cycles
are popped, the exposed arrows still include cycles, so a transfinite
version of the algorithm is used.
Let $\cW := (W_1,W_2,\dots)$, where each $W_i$ is a (finite or
infinite) sequence of vertices $W_i = (v_{i,j})_{1 \le j \le \ell_i}$.
Roughly speaking, we go over the vertices in $\cW$ one at a time
according to the order on the indices $(i,j)$ induced by the
lexicographical order.
vertices, where the ordering is described by lexicographic ordering in
the indices: $v_{ij} \preceq v_{i'j'} $ if $i \le  i'$ or if $i = i'$,
$j \le j'$. Observe that $\cW$ also defines an order of popping loops.
Namely,
When we arrive at a vertex $v=v_{i,j}$, we check whether it is included
in a finite cycle defined by the currently exposed arrows (such a cycle
is unique if it exists); if so, we pop it; otherwise, we do nothing.
Formally, we inductively define $c_{i,j}(x) \in \N$ (the color of the
exposed arrow at $x$ at stage $(i,j)$) as follows. Set $c_{1,1}(x) := 1$.
If we have defined $c_{i,j}=(c_{i,j}(x))_{x \in V}$ for some $(i,j)$,
then we define $c_{i,j+1}(x) := c_{i,j}(x) +\1_{x \in L_{i,j}}(x)$,
where $L_{i,j}$ is the finite cycle passing through $v_{i,j}$ in the
arrow configuration $(S^{c_{i,j}(x)}_x)_x$ (where $L_{i,j}=\emptyset$ if
$j>\ell_i$ or if no such cycle exists).
When moving to the next sequence $W_{i+1}$, define $c_{i+1,1} := \lim_{j
\to \infty} c_{i,j}$.
Finally, define $c_\infty := \lim_{i \to \infty} c_{i,1}$.
Note that $\{ L_{i,j} \}$ is the set of cycles popped, that
$c_\infty(x)$ is the final color of the arrow at $x$, and that
$\sigma=(S_x^{c_\infty(x)})_x$ is the final arrow configuration.

We say $\cW $ is a \textbf{legal order} with respect to a collection of
stacks $S$ if the following hold for every $v\in V$:
\begin{itemize}
\item $c_\infty(v)$ is finite (only finitely many loops passing through
$v$ are popped),
\item There are infinitely many $i$ such that $v$ appears in $W_i$.
\end{itemize}
The motivation behind the above definition is as follows: the first item
guarantees that the popping stabilizes (and that the above construction
is well defined);
The second will ensure that every cycle that can be popped is indeed popped.
We now make two claims:
\begin{description}
\item[Claim 1] Almost surely, there exists a legal order.
\item[Claim 2] For any two legal orders, the set of colored cycles
popped the same.
   In particular, any two legal orders produce the same final arrow
configuration $\sigma$.
\end{description}

Claim 1 is a consequence of the fact that Wilson's algorithm rooted at
infinity almost surely stabilizes.
Let $(u_1,u_2, \dots)$ be some ordering of the vertices of $G$.
Stacks of arrows can be used to generate random walks on $G$:
At each step the walker uses the first unused arrow at its current location.
Consider the resulting random walk from $u_1$.
Each loop created by this walk corresponds to a cycle that can be popped.
We define the sequence $W_1$ to be the of ordered sequence of vertices at which the walk closes
a loop (with multiplicities). 
Thus each vertex of $W_1$ will cause a cycle to be popped.
After popping these cycles in order, the loop erased walk consists of an
infinite directed path of exposed arrow starting at $u_1$.
We define $W_2$ to be that path, and note that visiting the vertices of
$W_2$ will not pop any cycles.
(This is still needed for the second legality condition.)

Having defined $W_1$ and $W_2$, we consider the walk from $u_2$ stopped
when it hits the path from $u_1$ (or to infinity if it does not hit the
path).
The vertices where this walk closes a loop form $W_3$.
After popping these cycles, the union of the two loop erased walks is a
forest, and all vertices of this forest are entered in $W_4$ in an
arbitrary order.
In general, the loops closed by the walk from $u_k$ will be noted in
$W_{2k-1}$, and the forest spanned by $u_1,\dots,u_k$ will be included
in $W_{2k}$. It is clear that this sequence of vertices describe Wilson's algorithm
and hence almost surely the first item in the definition of legal order
is satisfied since Wilson's algorithm stabilizes by transience. Since we
always scan all the vertices in the forest in $W_{2i+1}$,
each vertex $v$ appears in infinitely many odd indexed $W_i$s.

We now proceed Claim 2.
The proof method is similar to that in the finite case, see e.g.
\cite[Lemma 4.2]{lyons2017probability}, with some extra care to deal
with the infinite graph.
We provide it for completeness.
Consider two legal orders $W$ and $W'$
The idea is to use transfinite induction to show that every colored
cycle $C'$ that is popped when following $W'$ is also popped at some
point when following $W$.
For the base case, consider the very first cycle that is popped when
following $W'$.
All arrows of the colored cycle $C'$ must have color 1.
Consider the first cycle popped in $W$ that intersects $C'$.
Call that cycle $C$.
Since no previous popped cycle intersects $C'$, when $C$ is popped in
includes a vertex $v\in C'$ which still has color 1.
When $C$ is popped, all vertices of $C'$ have color 1, and so if $C$
includes any of them it must equal $C'$ (as a colored cycle!).

The inductive step is similar. The key observation is the following.
Suppose a vertex $v_{i,j}$ causes a colored cycle to be popped.
If we remove $v_{i,j}$ from the order, and remove the arrows of the
cycle from the stacks, then the resulting set of colored cycles popped
will not see any other change.
That is, applying the modified sequence to the modified stacks will see the
same cycles being popped.
The same holds if any number of such terms are dropped from $W$ and the
arrows of the corresponding cycles are removed from the stacks.

Consider now some colored cycle $C'$ that is popped in $W'$.
Let $\hat W'$ be the order $W'$ without all the vertices that lead to a
popped cycle strictly before $C'$.
Let $\hat S$ be the stacks without all the arrows involved in those
earlier cycles.
Note that these are all locally finite, so are well defined.
By the induction hypotheses, all earlier colored cycles are also popped
in $W$, so let $\hat W$ be the order $W$ without the vertices leading to
these cycles being popped there.
Now $C'$ is the first cycle popped in $\hat W'$ with the stacks $\hat S$.
As in the base case, $C'$ is also popped in the order $\hat W$ on $\hat S$.
By the above observation, $C'$ is also popped in $W$ acting on $S$.

We have shown that all colored cycles popped in $W'$ are also popped in
$W$, and by symmetry the converse also holds, proving Claim 2.


\medskip

Given these two claims, the graph factor map can now be easily described.
Given the stacks, choose any legal order and compute the final
configuration $\sigma$ of exposed arrows that it produces.
This configuration does not depend on the chosen legal order by Claim 2,
and hence this yields a graph factor.
Wilson's algorithm ensures that the unoriented arrows defined by
$\sigma$ has the law of the WUSF \cite[Proposition
10.1]{lyons2017probability}.
\end{proof}

\begin{remark}
	For our application to the loop $O(1)$ model, we use \cref{USFfactor} together with the known fact that the components of the WUSF are one-ended almost surely~\cite{BLPS_USF,AL_unimodular,H15a}. In fact, for our purposes, we could use \emph{any} forest (instead of the WUSF) which can be obtained as a graph factor of i.i.d.\ and whose components are infinite one-ended trees (see \cref{prop:tree_construction}).
   In this context, we mention that another natural candidate for the desired forest is
the free or wired \emph{minimal spanning forest}, which has a natural
description as a factor of i.i.d.
   These forests are defined as follows.
   Attach an independent Uniform$[0,1]$ random variable $U(e)$ to each
edge  $e$ in the graph.
   Declare an edge $e$ to be present in the free minimal spanning forest
if there is no cycle in $G$ such that $U(e)>U(e')$ for all $e' \neq e$
in the cycle.
   For the wired minimal forest, we additionally require that there is
no bi-infinite path containing $e$ such that $U(e) >U(e')$ for all $e'$
in the bi-infinite path.
   It is known that the wired minimal spanning forest has one-ended
components on a unimodular, extremal random rooted graph if there is no
infinite cluster at criticality \cite[Theorem 7.4]{AL_unimodular}.
   We do not know how to prove that there is no infinite cluster at
criticality in our setting (notably for the supercritical FK-Ising
model, except if $p=1$ and $G$ is nonamenable and quasi-transitive).
   We mention that one route which is sufficient is to prove that the
FK-Ising cluster is invariantly non-amenable as then we can use
\cite[Theorem 8.11]{AL_unimodular}.
   However this problem is believed to be hard, recently this was solved
for Bernoulli bond percolation \cite{hermon2021supercritical}.
   We also mention that it is not known whether the free minimal
spanning forest is connected in general (see \cite{lyons2006minimal} for
a discussion).
\end{remark}

\section{Loop $O(1)$ as a factor of i.i.d.}\label{sec:loop_factor}

In this section, we prove \cref{thm:main_wired} (wired case) and \cref{thm:main_free} (free case). 

We record here a result about obtaining an end faithful spanning tree as a factor of i.i.d.\ in unimodular amenable graphs.
\begin{thm}[\cite{benjamini1999group,AL_unimodular,Timar_one_ended}] \label{thm:timar}
Let $(G, \rho)$ be a unimodular random rooted graph which is invariantly amenable with finite expected degree of $\rho$. There exists an end faithful spanning tree of $(G, \rho)$ which can be obtained as a graph factor of i.i.d.
\end{thm}

\begin{proof}
It is known that if $(G, \rho)$ is invariantly amenable then there exists a spanning tree with at most two ends which can be obtained as a graph factor of i.i.d.\ This is proved in \cite[Theorem 8.9]{AL_unimodular} which extends \cite[Theorem 5.3]{benjamini1999group} (it is not mentioned in the results that the tree can be constructed as a factor, but is implicit in their proofs\footnote{This is also mentioned in a remark in Tim\'{a}r's paper \cite{Timar_one_ended}.}). This shows that $(G, \rho)$ is also at most two ended. On the event that $(G, \rho)$ is two ended, it is clear that the spanning tree obtained as a factor must be two ended. On the event that $(G, \rho)$ is one ended, the theorem follows from the main result of Tim\'{a}r~\cite{Timar_one_ended}. 
\end{proof}

\begin{proposition}\label{prop:tree_construction}
  Let $(G, \rho)$ be a unimodular random rooted graph with finite expected degree of~$\rho$ and let $\omega$ be an invariant percolation on it which almost surely has no two-ended components. Conditionally on $(G, \rho, \omega)$, let $\Xi:=(\Xi_v)_{v \in V(G)}$ be an i.i.d.\ collection of Uniform$[0,1]$ random variables.
  Then there exists a graph factor map $\varphi_2$ with
  \[(G, \rho, (\omega, T)) = \varphi_2(G, \rho, (\omega, \Xi)),\]
  such that $T$ almost surely satisfies the following: $T$ restricted to a finite cluster of $\omega$ is a spanning tree of that cluster; $T$ restricted to an infinite cluster of $\omega$ is a spanning forest of that cluster consisting of infinite one-ended trees. 
\end{proposition}

(The reason for the notation $\varphi_2$ will become clear later.)

\begin{proof}
Let $\omega_v$ be the edge cluster of a vertex $v$ in $\omega$ and let $\Xi_v$ be $\Xi$ restricted to $\omega_v$. Recall from \cref{unimod_properties} \cref{root_cluster} that $(\omega_\rho, \rho)$ is unimodular.
 We claim that there is a graph factor $\tilde \varphi_2$ such that 
 $
 \tilde \varphi_2(\omega_\rho, \rho , \Xi_\rho) = (\omega_\rho, \rho , T_\rho),
 $
where $T_\rho$ satisfies the criteria of the proposition: if $\omega_\rho$ is finite, $T_\rho$ is a.s.\ a spanning tree; if $\omega_\rho$ is infinite, $T_\rho$ is a.s.\ a spanning forest whose every component is infinite and one ended. Note that given $\tilde \varphi_2$, we can define  $\varphi_2$ by applying $\tilde \varphi_2$ separately to each vertex component (precisely because $\tilde \varphi_2$ is a graph factor and does not depend on the location of the root). We can thus conclude that $\varphi_2$ produces a $T$ as desired.

 Let us now define $\tilde \varphi_2$. If $\omega_\rho$ is finite, then it is a standard fact that we can sample a uniform spanning tree as a graph factor, call this map $\psi$.
Let $\cR$ be the event that every infinite cluster is recurrent and one ended. If $\cR$ has positive probability, $(G, \rho, \omega)$ conditioned on $\cR$ is a unimodular random rooted graph in which every infinite cluster is one ended and recurrent (\cref{unimod_properties}, \cref{reroot}), and the required factor map was constructed by Tim\'{a}r (\cref{thm:timar}); call it $\psi'$. Since $\omega$ does not have any two-ended components by assumption, $\cR^c$ is the event that every infinite cluster is transient (here we use implicitly that it almost never happens that some infinite clusters are recurrent and some are transient by \cref{unimod_properties} \cref{rec_one_ended}).
So now we can use \cref{USFfactor} to obtain the required factor map, call it $\psi''$, which has the desired one-endedness property (see
BLPS \cite[Theorem 10.1]{BLPS_USF} for Cayley graphs, Aldous-Lyons for
unimodular, bounded degree graphs \cite[Theorem
7.2]{AL_unimodular}, and Hutchcroft  \cite[Theorem 1]{H15a} for the most
general result).
Finally we define $\tilde \varphi_2((g,x,m))$ to be equal to $\psi$ if $(g,x)$ is finite, $\psi'$ if $(g,x)$ is recurrent and one ended, and $\psi''$ if $(g,x)$ is transient. This completes the construction of $\tilde \varphi_2$ and finishes the proof.
\end{proof}

%

\begin{proof}[Proof of \cref{thm:main_wired} (Wired case.)] We now turn to the proof of \cref{thm:main_wired}, i.e., the goal now is to describe a graph factor map $\varphi$ such that $$\varphi((G, \rho, \Xi))  = (G, \rho, \eta) $$ where $(G, \rho, \Xi)$ is an i.i.d.\ marked random rooted graph and given $(G, \rho)$, $\eta \sim \bP^{\w}_{G,x,y}$. Let $\Xi = (\Xi^{(1)},\Xi^{(2)},\Xi^{(3)})$ to be three i.i.d.\ collections. 
 Fix $\beta, h$ such that $x = \tanh(\beta), y = \tanh(h)$ and fix $p = 1-e^{-2\beta} $ and  $p_h = 1-e^{-2h}$ as prescribed by \cref{prop:coupling}.
 
 \medskip
 
We first handle the case $x \in [0, 1)$. Let $\omega^{\w} \sim \phi^{\w}_{G,p,p_h}$. Using \cref{thm:HS}, we know that $\omega^{\w} \sim \phi^{\w}_{G,p,p_h}$ is a graph factor of i.i.d., so that there is a graph factor map $\varphi_1$ such that $$ \varphi_1 ((G, \rho, \Xi^{(1)})) =(G, \rho, \omega^{\w}).$$
 Since $x<1$, we have that $\omega^{\w}$ is deletion tolerant, and we can thus conclude using \cref{unimod_properties} \cref{rec_one_ended}  that none of the components of $\omega^{\w}$ is two-ended almost surely.
Hence \cref{prop:tree_construction} allows us  to define a graph factor $\varphi_2$ such that
 $$
 \varphi_2((G, \rho,(\omega^{\w}, \Xi^{(2)}) )) = (G, \rho, (\omega^{\w}, T)),
 $$
 where $T$ almost surely satisfies the following properties. The restriction of $T$ to each infinite component of $\omega^{\w}$ is a spanning forest whose every component is infinite and one ended. The restriction of $T$ to each finite component is a spanning tree of that component.
Finally, we define a graph factor map $\varphi_3$ such that
 $$
 \varphi_3((G, \rho, (\omega^{\w}, T, \Xi^{(3)}))) = (G, \rho, \eta).
 $$
 This is obtained by extracting from $T$ a locally finite generating set for the wired cycle space of $\omega^{\w}$  using \cref{lem:gen_set_forest} and then defining $\eta$ as in \eqref{eq:U} (taking the i.i.d.\ variables from $\Xi^{(3)}$ and noting crucially that the elements in the generating set are indexed by edges of $G$).
By composing $\varphi_1, \varphi_2, \varphi_3$ in an obvious manner, we obtain the graph factor map $\varphi$. The fact that $(G, \rho, \eta)$ has the required distribution is a consequence of \cref{prop:loop_a.s._conv,prop:limit_projection}. 
\medskip

Now we turn to the case $x=1$. In this case, the edge set of $\omega^{\w}$ coincides with that of $G$. Thus the case when $(G, \rho)$ is a.s.\ not two ended is already handled by the previous argument (here we only need to compose the maps $\varphi_2, \varphi_3$). Now suppose that $(G, \rho)$ is two ended with positive probability and condition on it. Using \cref{thm:timar}, we obtain an end faithful two-ended spanning tree $T$ of $(G, \rho)$ as a graph factor of i.i.d. Now if $y>0$, using \cref{lem:tree_end_dense}, $T^*$ is end dense almost surely. Hence we can use the second item of \cref{lem:gen_set_forest} to obtain a locally finite generating set for the wired cycle space as a factor. The last remaining case is when $y=0$ and $(G, \rho)$ is two ended with positive probability, in which case we claim that $\eta$ is not a graph factor of i.i.d. The argument for this is independent of everything else and is proved in the following proposition.



\begin{proposition}\label{prop:two_ended}
Let $(G, \rho)$ be a unimodular random rooted graph which is two ended with positive probability. Then the wired uniform even subgraph of $(G, \rho)$  is \textbf{not} a graph factor of i.i.d.
\end{proposition}
\begin{proof}
We may assume without loss of generality that $(G,\rho)$ is ergodic, so that in particular it is almost surely two ended.
Let $U$ be the wired uniform even subgraph of $(G, \rho)$. Using \cref{lem:ergodic}, it suffices to show that $(G,\rho,U)$ is not ergodic. 

We define a non-constant random variable $X$ which is invariant to re-rooting.
This contradicts the ergodicity of $(G,\rho)$.
Call a set of edges an \text{end-separator} if removing it creates two infinite components.
A minimal end-separator is such a set which is minimal with respect to inclusion.
In this case, $G \setminus F$ is disconnected and has exactly two connected components. 
Let $F$ be a minimal end-separator.
Define $X$ to be $|U \cap F| \mod 2$.
Let us show that $X$ does not depend on the choice of $F$.
Let $F'$ be another minimal end-separator.
Now, $U$ can be written as an edge-disjoint union of finite cycles and bi-infinite paths, so it suffices to show that any finite cycle or bi-infinite path $C$ satisfies $|C \cap F| \equiv |C \cap F'| \pmod 2$.
It is not hard to check that $|C \cap F| \equiv 1 \pmod 2$ if and only if the intersection of $C$ with each of the two connected components of $G \setminus F$ is infinite. Since the latter clearly does not depend on the end-separator $F$ in a two-ended graph, we conclude that $X$ is well defined and invariant to re-rooting.

It remains to show that $X$ is a non-constant random variable.
Let $P$ be a bi-infinite path in $G$, whose ends corresponding to the two ends of $G$. Note that $U \mapsto U \Delta P$ is measure preserving (i.e., $U \Delta P$ has the same distribution as $U$) and that it maps $X$ to $1-X$.
Thus, $X$ is a Bernoulli$(1/2)$ random variable.
\end{proof}

This completes the proof of \cref{thm:main_wired}.
\end{proof}

\smallskip
\begin{proof}[Proof of \cref{thm:main_free} (Free case.)]
We will use the notations of the factor maps used in the wired case.
  Let us first sample $\omega^{\f} \sim \phi^{\f}_{G,p,p_h}$ as a graph factor of i.i.d.\ by applying \cref{thm:HS} (analogous to $\varphi_1$ as in the wired case).

  \paragraph{Case (a)}
  If $y>0$ then we know by \cref{lem:wired=free,lem:loopO1end_dense} that each infinite cluster is end-dense and hence the wired and free uniform even subgraphs have the same law.
  Thus, the same arguments as in the wired case yield that $\bP^{\f}_{G,x,y}$ is a factor of i.i.d.\ (with the same maps $\varphi_2, \varphi_3$).

  \paragraph{Case (b)}
  We may assume that $y=0$ (otherwise refer to Case (a)).
  If $(G, \rho)$ is invariantly amenable, then we can use \cref{unimod_properties} to say that $(\omega^{\f}_\rho, \rho)$ is also invariantly amenable.
  We can then apply \cref{thm:timar} to find an end faithful spanning tree as a graph factor of i.i.d.\
  We can apply \cref{lem:gen_set_free} to get a locally finite generating set for $\cE^{\f}$.
  Taking each element with probability 1/2 yields the desired factor.


  \paragraph{Case (c)}
  Suppose now that $\omega^{\f}$ does not have infinitely many geodesic cycles through any vertex almost surely.
  Let $\cC$ be the collection of all geodesic cycles in $(G,\rho)$, which by assumption is locally finite.
  We claim that $\cC$ generates the free cycle space $\cE^{\f}$.
  Indeed, suppose some cycle $C\in\cE^{\f}$ is not in the span of $\cC$.
  Let $C$ be such a cycle of minimal length.
  Since $C$ is not itself geodesic, some pair of vertices in $C$ are connected by a shorter path.
  This allows us to write $C$ as a sum of two shorter cycles, which must be in the span of $\cC$, contradicting the assumption that such $C$ exists. This shows that $\cC$ generates the free cycle space $\cE^{\f}$.
  Finally, taking each element with probability 1/2 yields the desired factor.


  \paragraph{Map case}
  If $M$ is planar map, we can take the collection of all finite degree faces of $\omega^{\f}$ as a locally finite generating set (this replaces the step of finding such a generating set using spanning trees in the wired case). Indeed, it is a standard fact that every cycle in a planar map is the sum of all the faces in the finite component enclosed by it.
  As before, we can assign a random parity to each cycle as a factor of i.i.d.\ because they are finite.

\medskip
This finishes the proof of \cref{thm:main_free}. 
\end{proof}

\begin{remark}
If $G$ is amenable and vertex transitive, then it is a result of Raoufi \cite{raoufi2020translation} that the free and wired Loop $O(1)$ measures coincide away from criticality (notably for low temperatures). Thus, in this case, we can alternatively conclude that the free Loop $O(1)$ is a graph factor of i.i.d.\ using \cref{thm:main_wired}.
\end{remark}

\begin{remark}\label{rmk:nonunimodular}
  If $(G, \rho)$ is vertex transitive and nonunimodular, then the steps where the proof of \cref{thm:main_wired} fails are as follows. If the FK-Ising clusters are transient, then we do not know if the WUSF is one-ended (see \cite{lyons2008ends} for a condition of uniform transience which ensures this). In the recurrent case, we would need a version of Tim\'ar's result (we believe this case is vacuous). We point out that planar, one-ended, quasi-transitive graphs are always unimodular (see \cite[Theorem 8.25]{lyons2017probability}).

  In the free case, the arguments for planar maps and for graphs with finitely many geodesic cycles through each $v$ do not require the unimodularity condition, and the result holds.
\end{remark}


\section{Perspectives and open questions}\label{sec:open}

\begin{question}\label{qn:recurrentUST}
Suppose $(G, \rho)$ is a unimodular random rooted graph which is recurrent almost surely. Show that the wired uniform spanning tree is a graph factor of i.i.d. 
\end{question}

This is open even in the case of $\Z^2$.
We believe that for $\Z^2$, the result of Lis and Duminil--Copin~\cite{duminil2019double} on double random currents and its connection with dimers could be a potential approach. However a more robust approach is more desirable, so we do not pursue that in this article.

\medskip

We now turn to the condition involving geodesic cycles in \cref{thm:main_free}. One situation where it is satisfied is if a sample from supercritical FK-Ising satisfies a weak version of Gromov hyperbolicity which we now define. Given three vertices $x,y,z$, let $T_{xyz}$ be the union of $\gamma_{xy}, \gamma_{yz}$ and $\gamma_{zx}$, where $\gamma_{uv}$ is a geodesic joining $u$ and $v$. This is called a geodesic triangle. Such a triangle is $\delta$-thin if the distance between any vertex in $\gamma_{yz}$ and $\gamma_{xy} \cup \gamma_{xz}$ is at most $\delta$, and the same is true for the other permutations of $\{x,y,z\}$.  A graph is anchored Gromov hyperbolic if for every $v$ there exists a $\delta$ (which may depend upon $v$) such that any geodesic triangle through $v$ is $\delta$-thin. It is straighforward to see that if $\omega$ is almost surely anchored Gromov hyperbolic, then the last condition in \cref{thm:main_free} is satisfied: indeed an arbitrarily long geodesic cycle can be used to create a geodesic triangle which is not $\delta$-thin for arbitrarily large $\delta$. 
It is not unreasonable to believe that if $G$ is Gromov hyperbolic, then an FK-Ising sample for a supercritical value of $p$ is anchored Gromov hyperbolic, but we do not pursue this in this article.

\begin{question}\label{qn:geodesic}
  Suppose $p_c(G)<p<1$.
  Could $\phi^{\f}_{G,p,0}$ include infinitely many geodesic cycles through the root? Here $p_c(G)$ is the critical probability for FK-Ising in $G$.
\end{question}

As discussed above, answering the following question will settle \cref{qn:geodesic} in the Gromov hyperbolic setting.

\begin{question}\label{qn:anchored}
  Suppose $G$ is Gromov hyperbolic. Show that for $p>p_c$, almost surely every infinite cluster of $\phi^{\f}_{G,p,0} $ is anchored Gromov hyperbolic.
\end{question}

Our results rely on a connection to FK-Ising which is only known for $x,y\in [0,1]$. This raises the following question:

\begin{question}
Suppose $\max\{x,y\}>1$. Are the free or wired Loop $O(1)$ measures well defined on infinite graphs, and if so are they graph factors of i.i.d.\ ?
\end{question}
In some cases simple duality relations can yield a positive answer to the above question.
If all vertices of $G$ have even degree then complement of the open edges yield a Loop $O(1)$ model with parameters $1/x,y$, so we can apply our theorem in this case if $x>1, y\in [0,1]$. If all the vertices have odd degree, then the complement of the open vertices and open edges yield a Loop $O(1)$ model with parameter $1/x,1/y$. Thus we can apply our theorem if $x \ge 1, y \ge 1$.

\bibliographystyle{amsplain}
\bibliography{loop}

\end{document}